\documentclass[reqno,11pt]{amsart}
\usepackage{fullpage}
\usepackage[backref,colorlinks]{hyperref}
\usepackage{amssymb}
\usepackage{amsmath}
\usepackage{amsthm}
\usepackage{bbm}
\usepackage{enumerate}
\usepackage{lineno}

\usepackage{changes}
\definechangesauthor[color=blue]{mk}
\definechangesauthor[color=red]{hiep}
\definechangesauthor[color=magenta]{mp}

\newtheorem{theorem}{Theorem}

\newtheorem{lemma}[theorem]{Lemma}
\newtheorem{proposition}[theorem]{Proposition}
\newtheorem{corollary}[theorem]{Corollary}
\newtheorem{fact}[theorem]{Fact}
\newtheorem{remark}[theorem]{Remark}

\DeclareMathOperator{\Lip}{Lip}

\newcommand{\CC}{\mathbb{C}}
\newcommand{\EE}{\mathbb{E}}

\newcommand{\NN}{\mathbb{N}}
\newcommand{\PP}{\mathbb{P}}
\newcommand{\RR}{\mathbb{R}}
\newcommand{\ZZ}{\mathbb{Z}}

\newcommand{\sub}{\mathrm{sub}}

\newcommand{\bw}{\boldsymbol{w}}
\newcommand{\bv}{\boldsymbol{v}}
\newcommand{\bu}{\boldsymbol{u}}
\newcommand{\bz}{\boldsymbol{z}}
\newcommand{\bg}{\boldsymbol{g}}
\newcommand{\bi}{\boldsymbol{i}}
\newcommand{\bh}{\boldsymbol{h}}
\newcommand{\bX}{\boldsymbol{X}}
\newcommand{\bx}{\boldsymbol{x}}
\newcommand{\ba}{\boldsymbol{a}}

\newcommand{\vecone}{\boldsymbol{1}}

\newcommand{\vecf}{\boldsymbol{f}}
\newcommand{\vecg}{\boldsymbol{g}}

\newcommand{\cal}{\mathcal}

\newcommand{\cB}{{\cal B}}
\newcommand{\cF}{{\mathcal{F}}}
\newcommand{\cP}{{\mathcal{P}}}
\newcommand{\cQ}{{\mathcal{Q}}}

\newcommand{\cW}{{\cal W}}

\newcommand{\rB}{{\mathrm B}}
\newcommand{\rU}{{\mathrm U}}
\newcommand*\diff{\mathop{}\!\mathrm{d}}

\let\eps=\varepsilon

\title{Quasi-random words and limits of word sequences}
\author[Hiep Han]{Hi\d{\^e}p H\`an}
\address{Departamento de  Matem\'atica y Ciencia de la Computaci\'on, Universidad de Santiago de Chile, Las Sophoras 173, Santiago, Chile}
\email{hiep.han@usach.cl}
\author{Marcos Kiwi}
\address{Departamento de Ingenier\'ia Matem\'atica y Centro de Modelamiento Matem\'atico (CNRS IRL2807), Universidad de Chile, Santiago, Chile}
\email{mkiwi@dim.uchile.cl}
\author{Mat\'ias Pavez-Sign\'e}
\address{Centro de Modelamiento Matem\'atico (CNRS IRL2807), Universidad de Chile, Santiago, Chile}
\thanks{The first author was supported by the FONDECYT Regular grant 1191838.
  The second author gratefully acknowledges the support of
  grant GrHyDy ANR-20-CE40-0002 and ANID via grant PIA AFB170001.
  The third author was supported by ANID Doctoral scholarship ANID-PFCHA/Doctorado Nacional/2017-21171132 and ANID via grant PIA AFB170001.}
\email{mpavez@dim.uchile.cl}

\date{}
\begin{document}

\begin{abstract}
Words are sequences of letters over a finite alphabet. We 
study two intimately related topics for this object:
quasi-randomness and limit theory.  With respect to the first topic
we investigate the notion of uniform distribution of letters over
intervals, and in the spirit of the famous Chung--Graham--Wilson theorem
for graphs we provide a list of word properties which are equivalent
to uniformity.  In particular, we show that uniformity is equivalent
to counting 3-letter subsequences.

Inspired by graph limit theory we then
investigate limits of convergent word sequences, those in which all
subsequence densities converge.  We show that convergent word
sequences have a natural limit, namely Lebesgue measurable functions
of the form $f:[0,1]\to[0,1]$.  Via this theory we show that every
hereditary word property is testable, address the problem of finite
forcibility for word limits and establish as a byproduct a new model
of random word sequences.

Along the lines of the proof of the existence of word limits, we can
 also establish the existence of
limits for higher dimensional structures. In particular, we obtain an
alternative proof of 
  the result by
  Hoppen, Kohayakawa, Moreira, R\'ath and Sampaio [{\it J. Combin. Theory Ser. B 103(1):93--113, 2013}] establishing the existence
  of permutons.

\end{abstract}
\maketitle


\section{Introduction}
%
Roughly speaking, quasi-random structures
are deterministic objects which share
many characteristic properties of their random counterparts. 
Formalizing this concept has turned out to be tremendously fruitful
in several areas, among others, number theory, graph theory, extremal
combinatorics, the design of algorithms and complexity theory.  This often follows from
the fact that if an object is quasi-random, then it immediately enjoys many other properties satisfied by its random counterpart.

Seminal work on quasi-randomness concerned graphs~\cite{CGW,Rodl86,Thomason}.
Subsequently, other combinatorial objects were considered,
which include subsets of~$\mathbb Z_n$~\cite{CGZn,Gowers},
hypergraphs~\cite{aigner2018quasirandomness,CGhyper,gowers_hyper,Towsner}, finite groups~\cite{gowers_groups}, and permutations~\cite{Cooper}. 
Curiously, in the rich history of quasi-randomness,  \emph{words}, i.e., sequences of letters from a
finite alphabet, one of the most basic combinatorial object with many applications, do not seem to have been explicitly investigated. We 
overcome this apparent neglect, put forth a notion of quasi-random words
and show it is equivalent to several other properties.
\medskip
In contrast to the classical topic of quasi-randomness, the research
of limits for discrete structures was launched rather
recently by Chayes, Lov\'asz, S\'os, Szegedy and Vesztergombi~\cite{BCLSV08,LS1}, and has become a very
active topic of research since. Central to the area is the notion of
convergent graph sequences $(G_n)_{n\to\infty}$, i.e., sequences of
graphs which, roughly speaking, become more and more ``similar'' as
$|V(G_n)|$ grows.
For convergent graph sequences, Lov\'asz and Szegedy~\cite{LS1} show
the existence of natural limit objects, called \emph{graphons},
   endow the space of these structures with a metric and establish
  the equivalence of their notion of convergence and convergence on such
  a metric.
Among many other consequences, it follows that quasi-random
graph sequences, with edge density $p+o(1)$, converge to the constant
$p$ graphon.


\medskip
In this paper, we continue the lines of previously mentioned
investigations and study quasi-randomness for words and limits of
convergent word sequences.  Not only in the 
literature of quasi-randomness but also in the one concerning limits of
discrete structures, explicit investigation of this fundamental object
has not been considered so far.

\section{Main contributions}
A word $\bw$ of length $n$ is an ordered sequence $\bw=(w_1,w_2,\dots,
w_n)$ of letters $w_i\in\Sigma$ from a fixed size alphabet $\Sigma$.
For the sake of presentation, unless explicitly said otherwise,
we restrict our discussion to the two
letter alphabet $\Sigma=\{0,1\}$, but most of our results and their proofs have
straightforward generalizations to finite size alphabets.

\subsection{Quasi-random words}
Concerning  quasi-randomness for words, our central notion is that of uniform distribution of letters over  intervals. 
Specifically, a word $\bw=(w_1\dots w_n)\in\{0,1\}^n$ is called \emph{$(d,\eps)$-uniform} if  for every interval $I\subseteq[n]$
 we have\footnote{We write  $a\pm x$ to denote a number contained in the interval $[a-x,a+x]$.}
\begin{align}\label{eq:defqr}
  \sum_{i\in I}w_i=|\{i\in I\colon w_i=1\}|=d |I|\pm \eps n.
\end{align}
 We say that $\bw$ is \emph{$\eps$-uniform} if $\bw$ is $(d,\eps)$-uniform for some $d$. Thus, uniformity states that up to an error term of $\eps n$  the number  of 
 1-entries of $\bw$ in each interval~$I$ is roughly $d|I|$,
 a property which binomial random words with parameter $d$ satisfy with 
high probability. This notion of uniformity has been studied by Axenovich, Person and Puzynina in~\cite{APP13}, where 
a regularity lemma for words was established and applied to the problem of finding twins in words. 
In a different context, it has been studied  by Cooper~\cite{Cooper} who  gave a list
of equivalent properties. 
A word $(w_1,\dots, w_n)\in \{0,1\}^n$ can also be seen as the set  $W=\{i\colon w_i=1\}\subseteq \mathbb Z_n$ and from this point of view uniformity should be compared to 
the classical notion of quasi-randomness of subsets of $\mathbb Z_n$, studied by Chung and Graham in~\cite{CGZn} and extended to the notion of $U_k$-uniformity by Gowers in~\cite{Gowers}. 
With respect to this line of research we note that our notion of uniformity is strictly weaker than all of the ones studied in \cite{CGZn,Gowers}. 
Indeed, the weakest of them concerns $U_2$-uniformity and may be rephrased as follows:  $W\subseteq \mathbb Z_n$ has $U_2$-norm at most $\eps>0$ if
for all $A\subseteq \mathbb Z_n$ and all but $\eps n$ elements $x\in\mathbb Z$ we have $|W\cap (A+x)|= |W|\frac{|A|}n\pm \eps n$ where $A+x=\{a+x\colon a\in A\}$.
Thus, e.g., the word $0101\dots 01$ is uniform in our sense but its corresponding set does not have small $U_2$-norm.

Analogous to the graph case there is a counting property related to uniformity.
Given a word $\bw=(w_1\dots w_n)$
  and a set of indices $I=\{i_1,\dots,i_\ell\}\subseteq [n]$, where $i_1<i_2<\dots<i_\ell$, 
  let $\sub(I,\bw)$ be the length
  $\ell$ subsequence $\bu=(u_1\dots u_\ell)$
  of $\bw$ such that $u_{j}=w_{i_j}$.
We show that uniformity implies adequate subsequence count, i.e., for any fixed~$\bu$ the number of subsequences equal to~$\bu$ in a large
uniform word~$\bw$, denoted by $\tbinom{\bw}\bu$, is roughly as expected from a random word with same density of 1-entries as~$\bw$.
It is then natural to ask whether the converse also holds and one of our main results concerning quasi-random  words states  that 
 uniformity is indeed already  enforced by  counting of
subsequences of length three. If we let $\|\bw\|_1=\sum_{i\in[n]}w_i$ denote the number of 1-entries in~$\bw$, then our result reads as follows.
\begin{theorem}\label{thm:quasirandom}\mbox{}
For every $\eps>0$, $d\in[0,1],$ and $\ell\in\NN$, there is an $n_0$ such that for all $n>n_0$ the following holds. 
\begin{itemize}
\item 
  If $\bw\in\{0,1\}^n$ is $(d,\eps)$-uniform, then for each
  $\bu\in\{0,1\}^\ell$
  \[
  \tbinom{\bw}\bu=d^{\|\bu\|_1}(1-d)^{\ell-\|\bu\|_1}\tbinom n\ell\pm 5\eps n^\ell.
  \]
\item Conversely, if
  $\bw\in\{0,1\}^n$ is such that for all $\bu\in\{0,1\}^3$ we have \[\tbinom {\bw}\bu=d^{\|\bu\|_1}(1-d)^{3-\|\bu\|_1}\tbinom n3\pm \eps n^3,\]
then $\bw$ is $(d,42\eps^{1/3})$-uniform.
\end{itemize}
\end{theorem}
 Note that in the second part of the theorem the density of 1-entries is implicitly given. This is because  
  $\binom \bw{(111)}=\binom{\|\bw\|_1}3$, and therefore the condition $\binom \bw{(111)}\approx d^3\binom n3$ implies that  $\|\bw\|_1\approx dn$.
We also note that  length three subsequences in the theorem  cannot be replaced by length two subsequences  and in this sense the result is best possible.
Indeed, the word $(0\dots01\dots10\dots 0)$ consisting of $(1-d)\frac n2$ zeroes followed by $dn$ ones followed by $(1-d)\frac n2$ zeroes 
contains the ``right'' number of every length two subsequences without being uniform. 

We also study a property called Equidistribution and show that it is equivalent to uniformity.
Together with Theorem~\ref{thm:quasirandom} (and its direct consequences)
and a result from Cooper~\cite[Theorem 2.2]{Cooper}
this yields a list of  equivalent properties
stated in Theorem~\ref{thm:quasirandom2}.
To state the result  let $\bw[j]$  denote the $j$-th letter of the word~$\bw$.
Furthermore, by the Cayley digraph
{$\Gamma(\bw)$}
of a word $\bw=(w_1,\dots,w_n)$
we mean the digraph on the vertex set $\ZZ_{2n}$\footnote{Choosing the vertex set to be $\ZZ_{2n}$ instead of $\ZZ_n$ avoids the graph  having loops.} in which $v$ is connected to $(v+i)\pmod{2n}$ for any~$i$ with $w_i=1$. Given a word {$\bu\in\{0,1\}^{\ell}$},
a  sequence of vertices $(v_1,\dots,v_{\ell+1})$ is an induced {$\bu$-walk} in
{$\Gamma(\bw)$}  if the numbers $i_1,\dots,i_\ell\in[n]$
defined by $v_{k+1}=v_k+i_k \pmod n$ satisfy $i_1<\dots<i_\ell$ and
for each $k\in [\ell]$ the pair $(v_k,v_{k+1})$ is an edge in~$\Gamma(\bw)$ if and only if $u_k=1$.
{Note that the number of induced $\bu$-walks in $\Gamma(\bw)$
  is precisely $2n\binom{\bw}{\bu}$.}


\begin{theorem}\label{thm:quasirandom2}
For a sequence $(\bw_n)_{n\to\infty}$  of words  $\bw_n\in\{0,1\}^n$ such that $\|\bw_n\|_1=dn+o(n)$ for some $d\in[0,1]$, the following are equivalent:
\begin{itemize}
\item (Uniformity) $(\bw_n)_{n\to\infty}$  is $(d,o(1))$-uniform.

  \smallskip
\item (Counting) For all  $\ell\in\NN$ and all  $\bu\in\{0,1\}^\ell$ we have
  \[\tbinom {\bw_n}\bu=d^{\|\bu\|_1}(1-d)^{\ell-\|\bu\|_1}\tbinom n\ell+ o(n^\ell).\]
\item (Minimizer) For  all $\bu\in\{0,1\}^3$ we have
  \[\tbinom {\bw_n}\bu=d^{\|\bu\|_1}(1-d)^{3-\|\bu\|_1}\tbinom n3+ o(n^3).\]
\item\label{it:expalpha} (Exponential sums)
  For any fixed  $k\in\NN$, $k\neq 0$, we have
  \[\textstyle\frac 1n\sum_{j\in[n]}\bw_n[j]\cdot\exp\left(\frac{2\pi  i}n k j\right)=o(1).\]
\item\label{it:intf} (Equidistribution) For every  Lipschitz function $f:\mathbb R/\mathbb Z\to\mathbb C$ 
  \[\textstyle\frac1n\sum_{j\in[n]}\bw_n[j]\cdot f(\tfrac jn)=d \int_{\mathbb R/\mathbb Z}f+o(1).\]
\item\label{it:increasingpaths}
  (Cayley graph) For  all $\bu\in\{0,1\}^3$ the number of induced $\bu$-walks in $\Gamma(\bw_n)$ is
  \[ d^{\|\bu\|_1}(1-d)^{3-\|\bu\|_1}2n\tbinom n3+ o(n^4). \]
\end{itemize}
\end{theorem}
We will say that a word sequence is
  \emph{quasi-random} if it satisfies one of (hence all) the properties of
  Theorem~\ref{thm:quasirandom2}.

\subsection{Convergent word sequences and word limits}
Over the last two decades it has been recognized that
  quasi-randomness and limits of discrete structures are intimately
  related subjects.  Being interesting in their own right, limit
  theories have also unveiled many connections between various
  branches of mathematics and theoretical computer science. Thus, as a
  natural continuation of the investigation on quasi-randomness,
  we study convergent word sequences
  and their limits, a topic which, to the best of our
  knowledge, has only been briefly mentioned by
  Szegedy~\cite{SzegedyICM18}.

The notion of convergence
    we consider is specified in terms of convergence of subsequence densities.
    Given $\bw\in\{0,1\}^n$ and~$\bu\in\{0,1\}^\ell$, let $t(\bu,\bw)$
    be the density of occurrences of~$\bu$ in $\bw$, i.e.,
    \[
    t(\bu,\bw)=\tbinom {\bw}\bu\tbinom{n}{\ell}^{-1}.
    \] 
Alternatively, if we define $\sub(\ell,\bw):=\sub(I,\bw)$ for~$I$ uniformly chosen among all subsets of $[n]$ of size $\ell$,  then $t(\bu,\bw)=\PP(\sub(\ell,\bw))=\bu)$.

A sequence of words $(\bw_n)_{n\to\infty}$ is called \emph{convergent} if for every finite word~$\bu$ the sequence $\big(t(\bu,\bw_n)\big)_{n\to\infty}$ converges. 
In what follows, we will only consider sequences of words such that the length of the words tend to infinity. This, however, is not much of a restriction since  
convergent word sequences with bounded lengths  must be constant eventually
and limits considerations for these sequences are simple.\footnote{Word sequences with bounded lengths 
contain a subsequence of infinite length which is constant and due to convergence all members of the original sequence must agree with this constant eventually. }

We show that convergent word sequences have  natural   limit objects, which turn out to be  Lebesgue measurable functions of the form $f:[0,1]\to[0,1]$. 
Formally, write $f^1=f$ and $f^0=1-f$ for a function $f:[0,1]\to[0,1]$ and
 for  a word $\bu\in\{0,1\}^\ell$ define
\begin{align}\label{eq:subseqdensity}t(\bu,f)=\ell!\int_{0\leq x_1<\dots<x_\ell\leq1}\prod_{i\in[\ell]}f^{u_i}(x_i)\diff x_1\dots \diff x_\ell.\end{align}
We say that \emph{$(\bw_n)_{n\to\infty}$ converges to $f$} and that $f$ is the \emph{limit} of $(\bw_n)_{n\to\infty}$, if for every word~$\bu$ we have\[\lim_{n\to\infty} t(\bu,\bw_n)=t(\bu,f).\] 
In particular, $(\bw_n)_{n\to\infty}$ is convergent in this case.
Furthermore, let $\cW$ be the set of all Lebesgue measurable functions of the form $f:[0,1]\to[0,1]$ in which, moreover, functions are identified 
when they are equal 
almost everywhere.
We show that each  convergent word sequence converges to a unique $f\in \cW$ and that, conversely, for each $f\in \cW$ there is a word sequence which converges to $f$.
\begin{theorem}[Limits of convergent word sequences]\label{thm:limits}\mbox{}
\begin{itemize}
\item For each convergent word sequence $(\bw_n)_{n\to\infty}$ there is an  $f\in\cW$ such that $(\bw_n)_{n\to\infty}$ converges to $f$.
Moreover, if  $(\bw_n)_{n\to\infty}$ converges to $g$ then $f$ and $g$ are equal almost everywhere.
\item Conversely, for every  $f\in\cW$ there is a word sequence $(\bw_n)_{n\to\infty}$ which converges to $f$.
\end{itemize}
\end{theorem}
Theorem~\ref{thm:limits} can be phrased in topological terms as follows. Given a word $\bu$, one can think of~$t(\bu,\cdot)$ as a function from $\mathcal W$ to $[0,1]$. Then, endow $\mathcal W$ with the initial topology with respect to the family of maps $t(\bu,\cdot)$, with $\bu\in\{0,1\}^\ell$ and $\ell\in\mathbb N$, that is, the smallest topology that makes all these maps continuous.
We show that this topology is actually metrizable and, moreover, compact (thereby proving Theorem~\ref{thm:limits}).

\medskip
The overall approach we follow is in line with what has been done for graphons~\cite{LS1} and permutons~\cite{HOPPEN}.  Nevertheless, there
are important technical differences, specially concerning the (in our
case, more direct) proofs of the equivalence between distinct notions
of convergence which avoid compactness arguments.
Instead, we rely on Bernstein polynomials and their properties as used in the (constructive) proof the Stone--{Weierstrass} approximation theorem.
 

\medskip
In contrast with other
technically more involved limit theories,
say the ones concerning graph sequences~\cite{LS1} and permutation
sequences~\cite{HOPPEN}, 
the simplicity of the underlying
combinatorial objects we consider (words) yields concise arguments,
elegant proofs, simple limit objects, and requires the introduction of
far fewer concepts.  Yet despite the technically comparatively 
simpler theory, many interesting aspects common to other structures
and some specific to words appear in our investigation.  As an
illustration, we work out the implications for testing of the class of
so-called \emph{hereditary} word properties and address the question
concerning \emph{finite forcibility} for words, i.e., 
which word limits are completely determined by a finite number of
prescribed subsequence densities.

\subsection{Testing hereditary word properties}
The concept
of self-testing/correcting programs was introduced by
Blum et al.~\cite{bk95,blr90} and greatly
expanded by the concept of graph property testing proposed by
Goldreich, Goldwasser and Ron~\cite{GGR98} (for an in depth coverage of the
property testing paradigm, the reader is referred to the book by
Goldreich~\cite{Goldreich17}). An insightful connection between
testable graph properties and regularity was established by
Alon and Shapira~\cite{AS05} and further refined in~\cite{AFNS09,AS08}.
It was then observed that similar and related results
can be obtained via limit theories (for the case of testing graph properties,
the reader is referred to~\cite{LStesting}, and for the case of (weakly) testing
permutation properties, to~\cite{jkmm11}).
Thus, it is not surprising that analogue results can be established for
word properties.
On the other hand, it is noteworthy that such consequences can be 
obtained very concisely and elegantly.

We next state our main result concerning testing word properties.
Formally, for $\bu,\bw\in\{0,1\}^n$ let $d_1(\bw,\bu)=\frac 1n\sum_{i\in [n]}|w_i-u_i|$.
A \emph{word property} is simply a collection of words.
A word property~ $\cP$ is said to be \emph{testable} if there is another
word property $\cP'$ (called \emph{test property for $\cP$})
satisfying the following conditions:
\begin{itemize}
\item[] (Completeness) For every 
  $\bw\in\cP$ of length $n$ and every $\ell\in [n]$,
  $\PP(\sub(\ell,\bw)\in\cP')\geq \tfrac{2}{3}$.

  \smallskip
\item[] (Soundness) For every $\eps>0$ there is an $\ell(\eps)\geq 1$ 
  such that if $\bw\in\{0,1\}^n$ with $d_1(\bw,\cP)=\min_{\bu\in\cP\cap\{0,1\}^n}d_1(\bw,\bu)\geq\eps$,
  then $\PP(\sub(\ell,\bw)\in\cP')\leq \tfrac{1}{3}$ for all
  $\ell(\eps)\leq\ell\leq n$.
\end{itemize}
{If completeness holds with probability $1$ instead of $2/3$
  one says that the property is testable with \emph{perfect completeness}.}
Variants of the notion of testability can be considered.
However, the one stated is sort of the most restrictive.
On the other hand, the notion can be strengthened by replacing the $2/3$
in the completeness part by $1-\eps$ and $1/3$ in the soundness part
by $\eps$. The notion can be weakened letting the test property
$\cP'$ depend on $\eps$.

A word property $\cP$ is called \emph{hereditary} if for each $\bw\in\cP$, 
every subsequence $\bu$ of $\bw$ also belongs to
$\cP$.
  
\begin{theorem}\label{theo:wordTesting}
  Every hereditary word property is testable with perfect completeness.
\end{theorem}
Since the notion of testability given above is very restrictive (it consists in sampling uniformly a constant number of characters from the word being tested) it straightforwardly yields efficient (polynomial time) testing procedures.

Hereditary properties can be characterized as collections $\cP_{\cF}$ of words that do not contain as subsequence any word in $\cF$ where
$\cF$ is a family of words ($\cF$ might even be infinite). For instance, given $\cP_1,...,\cP_k$ hereditary word properties, the collection
$\cP_{\text{col}}$ of words that can be $k$-colored (i.e., each of its
letters assigned a color in $[k]$) so that for all $c\in [k]$ the
induced $c$ colored sub-word is in~$\cP_c$ is an example of a hereditary
  word property.

\subsection{Finite forcibility}
Finite forcibility was introduced by Lov\'asz and S\'os~\cite{LovaszSos} while studying a generalization of quasi-random graphs. 
For an in depth investigation of finitely forcible graphons we refer to the work of Lov\'asz and Szegedy~\cite{LovaszSzegedy}. We say that $f\in\cW$ is \emph{finitely forcible}
if there is a finite list of words $\bu_1,\dots \bu_m$ such that any
function $h:[0,1]\to[0,1]$ which satisfies $t(\bu_i,h)=t(\bu_i,f)$ for
all $i\in[m]$ must agree with $f$ almost everywhere.  A direct
consequence of Theorem~\ref{thm:quasirandom} concerning quasi-random
words is that the constant functions are finitely forcible (by words
of length three).  We can generalize this result as follows:
\begin{theorem}\label{thm:forcible}
Piecewise polynomial functions are finitely forcible.
Specifically, 
if there is an interval partition $\{I_1,...,I_k\}$ of $[0,1]$, polynomials $P_1(x),...,P_k(x)$ of degrees $d_1,...,d_k$, respectively, and
  $f\in\cW$ is such that $f(x)=P_{i}(x)$ for all $i\in [k]$ and $x\in I_i$, then
  there is a list of words $\bu_1,\dots,\bu_m$, with $m\leq(k+1)^{2k^2(1+\max_i\deg P_i)}$ such that any function $h:[0,1]\to[0,1]$ which satisfies 
$t(\bu_i,h)=t(\bu_i,f)$ for all $i\in[m]$
 must agree with $f$ almost everywhere. 
\end{theorem}



\subsection{Extensions}\label{sec:generalizations}
We have considered quasi-randomness for words and limits of
convergent word sequences.
Our results are formulated for words over the alphabet $\{0,1\}$.
However, our results 
(except for the ones concerning testing word properties)
can be easily extended to any alphabet of finite size.
Also, note that a word of length $n$
  can be viewed as a $1$-dimensional $\{0,1\}$ array
  $A:[n]\to \{0,1\}$, which labels each element of $[n]$ with $0$ or $1$.
  Thus, a natural generalization of the $1$-dimensional binary word object
  is a $d$-dimensional $\{0,1\}$-array, $d$-array for short,
  $A:[n]^d\to\{0,1\}$.
  Our approach can also be generalized to handle $d$-arrays.
  Indeed, the natural extension to $d$-arrays of the notion of convergence
  of $1$-arrays yields a notion of convergent $d$-array sequence
  $(A_{n})_{n\to\infty}$, where $A_n:[n]^d\to\{0,1\}$ for all $n\in\NN$,
  whose limit is a Lebesgue measurable functions mapping
  $[0,1]^{d}$ to $[0,1]$ and where each such mapping is the limit of a
  convergent $d$-array sequence.

\subsection{Permutons from words limits}
Given $n\in\mathbb N$, we denote by $\mathfrak S_n$ the set of
permutations of order $n$ and $\mathfrak S=\bigcup_{n\ge 1}\mathfrak S_n$ the
set of all finite permutations.
Also, for $\sigma\in\mathfrak S_n$ and
$\tau\in\mathfrak S_k$ we let $\Lambda(\tau,\sigma)$ be 
the number of copies of $\tau$ in $\sigma$, that is, the number of $k$-tuples $1\le x_1<\dots<x_k\le n$ such that for
every $i,j\in[k]$
\[
\sigma(x_i)\le \sigma(x_j)\hspace{.5cm}\text{ iff }\hspace{.5cm}\tau(i)\le \tau(j).
\]
The density of copies of $\tau$ in $\sigma$, denoted by $t(\tau,\sigma)$,
is the probability that $\sigma$ restricted to a randomly chosen
$k$-tuple of $[n]$ yields a copy of $\tau$.
A sequence $(\sigma_n)_{n\to\infty}$ of permutations, with $\sigma_n\in\mathfrak S_n$ for each $n\in\mathbb N$, is said to be convergent if 
$\lim_{n\to\infty}t(\tau,\sigma_n)$
exists for every permutation $\tau\in\mathfrak S$. Hoppen et al.~\cite{HOPPEN} proved that every convergent sequence of permutations converges to a suitable analytic object called \textit{permuton}, which are probability measures on the Borel $\sigma$-algebra on $[0,1]\times[0,1]$ with uniform marginals, the collection
of which they denote by $\mathcal Z$, and also extend the map
$t(\tau,\cdot)$ to the whole of $\mathcal Z$.
Then, they define a metric $d_\Box$ on $\mathcal Z$ 
so that for all $\tau\in\mathfrak S$ the maps $t(\tau,\cdot)$ are
 continuous with respect to $d_\Box$.
They also show that $(\mathcal Z,d_\Box)$ is compact and, as a
consequence, establish that convergence as defined above and
convergence in $d_{\Box}$ are equivalent.
In particular, they prove that for every
convergent sequence of permutations $(\sigma_n)_{n\to\infty}$ there is
a permuton $\mu\in\mathcal Z$ such that $t(\tau,\sigma_n)\to
t(\tau,\mu)$ for all $\tau\in\mathfrak S$.
We give new proofs (see Proposition~\ref{perm:cauchy} and Theorem~\ref{thm:hoppen}) of these two results by using a more direct
approach based on Theorem~\ref{thm:limits}. 

\subsection{Organization} 
We discuss quasi-randomness in Section~\ref{sec:quasirandom},
proving Theorem~\ref{thm:quasirandom2} concerning the equivalent characterizations
of quasi-random words and the second part of
  Theorem~\ref{thm:quasirandom}, that uniformity is implied by the counting property of length three subsequences.
  The  first part of Theorem~\ref{thm:quasirandom}, which claims that uniformity entails the counting property of all subsequences,
 follows from the 
   more general Lemma~\ref{lem:tcount} from
  Section~\ref{sec:limitTheory}.
  
  In Section~\ref{sec:limitTheory} we develop the limit theory of
  convergent word sequences.
  Besides proving Theorem~\ref{thm:limits}, thus establishing the existence
  of word limits, among others, we also prove the uniqueness of such limits and that the initial topology of $\cW$ is metrizable and complete.  
  
Section~\ref{sec:testing} 
  is dedicated to the study of testable word
  properties, in particular to the proof of Theorem~\ref{theo:wordTesting} concerning testability of hereditary word properties.  
Finite forcibility is addressed
in Section~\ref{sec:forcible} where we prove Theorem~\ref{thm:forcible}
concerning forcibility of piecewise polynomial functions. The proof
also yields an alternative  proof of
the second part of Theorem~\ref{thm:quasirandom} which is moreover
formulated in the language of word limits, see Remark~\ref{rem:forcibleqr}.
Section~\ref{sec:permFromWords} is devoted to an alternative derivation of
two key results of Hoppen et al.~\cite{HOPPEN} about permutons.
In Section~\ref{sec:extensions}, we discuss generalizations of our
results to words over non-binary alphabets and
extensions to higher dimensional objects, specifically multi-dimensional arrays.
We conclude in Section~\ref{sec:final} with a brief discussion of
  potential future research directions.

\section{Quasi-randomness}\label{sec:quasirandom}
In this section we give the proof of the second part of Theorem~\ref{thm:quasirandom} and Theorem~\ref{thm:quasirandom2}.
{We start by establishing an inverse form of the
  Cauchy--Schwarz inequality which is used to prove the second part of
  Theorem~\ref{thm:quasirandom}, that
  controlling the density of subsequences of length three is enough to
  guarantee uniformity. An alternative  demonstration of  the second part of
  Theorem~\ref{thm:quasirandom} can be extracted from the proof of Theorem~\ref{thm:forcible} (see Remark~\ref{rem:forcibleqr}).
  
  Then, after recalling some basic facts and terminology about Fourier analysis and Lipschitz functions, we proceed to prove the equivalence of the  quasi-random properties listed in Theorem~\ref{thm:quasirandom2}.

\begin{lemma}\label{lem:inverseCS}
  If
  $\bg=(g_1,\dots,g_n),\bh=(h_1,\dots,h_n)\in\RR^n$ and $\eps \in(0,1)$
  are such that
  \[
  \langle\bg,\bh\rangle^2\geq \|\bg\|^2\|\bh\|^2-\eps n^3\|\bh\|^2,
  \]
  then all but at most $\eps^{1/3}n$ indices
  $i\in[n]$ satisfy
  $g_i=\tfrac{\langle\bg, \bh\rangle}{\langle\bh, \bh\rangle}h_i\pm\eps^{1/3}n.$
\end{lemma}
\begin{proof}
Let $\bz$ be the projection of $\bg$ onto the plane orthogonal to $\bh$, i.e., $\bz=\bg-\frac{\langle\bg,\bh\rangle}{\langle\bh,\bh\rangle}\bh$. As~$\bz$ and~$\bh$ are orthogonal, it follows that
\[\|\bg\|^2=\tfrac{\langle\bg,\bh\rangle^2}{\langle\bh,\bh\rangle^2}\|\bh\|^2+\|\bz\|^2=\tfrac{\langle\bg,\bh\rangle^2}{\|\bh\|^2}+\|\bz\|^2.\]
The assumption then yields
\begin{align}\label{eq:csi}\eps n^3\geq \|\bz\|^2=\sum_{i\in[n]}\left(g_i-\tfrac{\langle \bg, \bh\rangle}{\langle \bh,\bh\rangle}h_i\right)^2.\end{align}
Thus, the conclusion of the lemma must hold, otherwise $\|\bz\|^2>\eps^{1/3}n(\eps^{1/3}n)^2=\eps n^3$, contradicting~\eqref{eq:csi}.
\end{proof}

\begin{proof}[Proof (of the second part of Theorem~\ref{thm:quasirandom})]
Given $\eps>0$  let $n>n_0$ be sufficiently large. By a word containing $*$ we mean the family of words obtained by replacing~$*$ by $0$ or $1$, e.g.,
$\bu=(*u_2 u_3)$  denotes the family  $\{(0u_2u_3),(1u_2u_3)\}$. For a word $\bu$ containing $*$,
let $\binom \bw \bu=\sum_{\bu'} \binom \bw {\bu'}$ where the sum ranges over the family mentioned  above. 
 Given a word $\bw=(w_1\dots w_n)\in \{0,1\}^n$ which satisfies the assumption of the theorem we have
 \begin{align}\label{eq:starcount} \binom {\bw}{11*}\leq d^{2}\binom n3+ 2\eps n^3\quad\text{and}\qquad \binom {\bw}{*1*}+\binom {\bw}{1\!*\!*}\geq 2d\binom n3- 8\eps n^3.\end{align}
 We may also assume that  $d\geq \eps$, otherwise the first condition yields
 $\|\bw\|_1\leq 3\eps^{1/3} n$ due to 
 $\tbinom{\|\bw\|_1}{3}= \tbinom\bw{111}$ and   the result follows trivially.

Note that by assumption $\tbinom{\|\bw\|_1}{3}=d^3\tbinom{n}{3}\pm\varepsilon n^3$, implying that $\|\bw\|^3_1=d^3n^3\pm 7\varepsilon n^3$, whence
   $\|\bw\|_1=dn\pm 3\eps^{1/3}n$. Next, let $\bg=(g_1,\dots,g_n)$ where $g_\ell=\sum_{i\in[\ell]}w_i$  and let  $\bh=(1,2,\dots,n)$.
 It is easily seen that $\bw$ is $42\eps^{1/3}$-uniform if $g_{\ell}=d\ell\pm 21\eps^{1/3}n$ for every $\ell\in [n]$ and, since $g_n=\|\bw\|_1=dn\pm 3\eps^{1/3}n$, that the latter follows from
\begin{align}
\label{eq:defqr2} g_\ell=\frac{\langle \bg, \bh\rangle}{\langle \bh,\bh\rangle}\ell\pm 9\eps^{1/3}n\qquad\text{for every $\ell\in[n]$}.
\end{align}
To show~\eqref{eq:defqr2} note first that
\begin{align*}
   g_\ell^2 =|\{(i,j)\in [\ell]^2\colon w_{i}=w_{j}=1\}|\leq |\{(i,j)\in [\ell-1]^2\colon w_{i}=w_{j}=1,  i\neq j \}|+3(\ell-1)+1.
\end{align*}
Hence, up to an additive error of $3(\ell-1)+1$ the quantity $g_\ell^2$ is
twice the number of subsequences of {$\bw$ of the form $(11*)$ ending at $\bw$'s $\ell$-th letter.}
Summing over all $\ell\in[n]$ we obtain from~\eqref{eq:starcount}
\begin{align}\label{eq:normg}\|\bg\|^2=\sum_{\ell\in[n]}g_\ell^2\leq 2\binom{\bw}{11*}+\tfrac{3}{2}n^2\leq 2d^2\binom{n}3+5\eps n^3.\end{align} 
Consider next, for an $\ell\in[n]$, the family $S_\ell$ of {pairs $(i,j)\in [\ell-1]$, $i\neq j$, such that $(w_iw_j)$ is a subsequence of
  $(w_1,...,w_{\ell-1})$ and either $w_i=1$ or $w_j=1$.}
Then, we have $|S_\ell|\leq g_\ell\cdot\ell$, since there are at most
$g_\ell$ choices for $i$ and each such choice of $i$ gives rise to
$(i-1)+(\ell-i-1)\leq\ell$ choices for $j$.
On the other hand,  $\sum_{\ell\in[n]}|S_\ell|$ counts all subsequences of
$\bw$ of the form $(*1*)$ and $(1\!*\!*)$. 
Hence,~\eqref{eq:starcount} together with  $\bh=(1,2,\dots,n)$ yields
\[
\langle \bg,\bh \rangle^2
   =\Big(\sum_{\ell\in[n]}g_\ell\cdot \ell\Big)^2 \geq \Big(\sum_{\ell\in[n]}|S_\ell|\Big)^2 =\left(\binom\bw{*1*}+\binom\bw{1\!*\!*}\right)^2
   \geq  4d^2\binom{n}3^2-32\eps \binom n3n^3.
\]
As $\|\bh\|^2=\sum_{i\in[n]} i^2=\frac16 {n(n+1)(2n+1)}=2\binom n3+ \frac 32n^2-\frac n2$  from \eqref{eq:normg} we obtain 
\begin{align*}
\langle \bg,\bh\rangle^2- \|\bg\|^2\|\bh\|^2
  & \geq 4d^2\binom{n}3^2-32\eps \binom n3n^3- \left(2d^2\binom{n}3+5\eps n^3\right)\|\bh\|^2\\
  & \geq  2d^2\binom{n}3\left(\|\bh\|^2-\frac 32 n^2\right)- 16\eps  n^3\|\bh\|^2 - \left(2d^2\binom{n}3+5\eps n^3\right)\|\bh\|^2\\
  & \geq -22\eps n^3\|\bh\|^2.
\end{align*}
By Lemma~\ref{lem:inverseCS}
all but at most $(22\eps)^{1/3}n$ indices $i\in[n]$ satisfy 
$g_i=\frac{\langle \bg, \bh\rangle}{\langle \bh,\bh\rangle}i\pm(22 \eps)^{1/3}n.$
In particular, for every $\ell\in[n]$ there is such an index $i$ with $i=\ell\pm (22 \eps)^{1/3}n$. Thus 
\[
g_\ell=g_i\pm (22 \eps)^{1/3}n=\frac{\langle \bg, \bh\rangle}{\langle \bh,\bh\rangle}i\pm 2(22 \eps)^{1/3}n=\frac{\langle \bg, \bh\rangle}{\langle \bh,\bh\rangle}\ell\pm 3(22 \eps)^{1/3}n
\]
which shows~\eqref{eq:defqr2} and the second part of Theorem~\ref{thm:quasirandom} follows. 
\end{proof}

\begin{remark}
  The previous proof shows something stronger than what is claimed.
  Specifically, that instead of requiring the right count of  all subsequences of  length three it is  sufficient to have~\eqref{eq:starcount}, i.e., the correct upper bound for
the count of $(11*)$ and the correct lower bound for the sum of the count of $(*1*)$ and  $(1\!*\!*)$.
\end{remark}

We now turn our attention to Theorem~\ref{thm:quasirandom2} and recall here some facts from Fourier analysis on the circle. 
{Letting $\diff x$ correspond to the Lebesgue measure on the unit circle, for} $k\in\mathbb Z$,  the Fourier transform $\widehat{f}(k)$ of a function $f:\RR/\ZZ\to\CC$ is defined  by
	\[\widehat{f}(k)=\int_{\mathbb R/\mathbb Z}f(x)e^{-2\pi i k x}\diff x{.}\]
Given $N\in\NN$, the \textit{Fej\'er approximation} of order $N$ of $f$ is defined by
	\[\sigma_N f(x)= \sum_{|n|\le N}\Big(1-\frac{|n|}{N+1}\Big)\widehat{f}(n)e^{2\pi inx}.\]
Finally, we define the \textit{Lipschitz-norm} of $f$
 as $\|f\|_{\Lip}=\|f\|_\infty+\sup_{x\not=y}\frac{|f(x)-f(y)|}{d(x,y)}$, where $d(x,y)=\min\{1-|x-y|,|x-y|\}$ is the usual distance in $\mathbb R/\mathbb Z$.
\begin{lemma}[Proposition~1.2.12 from \cite{pinsky}]\label{lem:lip1} There is a constant $C>0$ such that for any Lipschitz function $f:\mathbb R/\mathbb Z\to \mathbb C$  and for every $M\ge 2$ one has
	\[ \|f-\sigma_Mf\|_\infty\le C \|f\|_{\Lip}\frac{\log M}{M}.\]
	\end{lemma}
\begin{lemma}[Theorem~1.5.3 from \cite{pinsky}]\label{lem:lip2}There is a constant $c>0$ such that for any Lipschitz function  $f:\mathbb R/\mathbb Z\to \mathbb C$ and for every $m\not=0$ one has
	\[|\widehat{f}(m)|\le \frac{c\|f\|_{\Lip}}{|m|}.\]
	\end{lemma}
We are now in the position to prove Theorem~\ref{thm:quasirandom2}.
\begin{proof}[Proof (of Theorem~\ref{thm:quasirandom2})]
  The equivalence between the Uniformity, Counting, {and Minimizer} properties follow from
  Theorem~\ref{thm:quasirandom}.
  The Cayley graph and Counting properties are the same property since
  there is a one-to-$n$
  correspondence between subsequences in $\bw_n$
  equal to $\bu$ and induced $\bu$-walks
  in $\Gamma(\bw_n)$.
  {To see this,
    simply note that $(v_1,\dots,v_{\ell+1})$ is an induced
    $\bu$-walk in $\Gamma(\bw_n)$ if and only if $(v_1+a,...,v_{\ell+1}+a)$ is
    an induced $\bu$-walk in $\Gamma(\bw_n)$, for all $a\in [n]$ (where arithmetic over vertices is modulo $n$).}
  The equivalence between the properties Uniformity and Exponential sums   was shown {by Cooper} in~\cite[Theorem~2.2]{Cooper}.\footnote{In~\cite{Cooper}, Cooper works in the context of subsets of $\mathbb Z_n$, calling a subset $S\subseteq \mathbb Z_n$ $\varepsilon$-balanced if $D(S)=\sup_{I\subseteq \mathbb Z_n}\left||S\cap I|-\frac{|S|}{n}|I|\right| \le \varepsilon n,$
  	where the supremum is taken over all the intervals $I\subseteq \mathbb Z_n$. Defining the $n$-letter word $W_S=w_1\dots w_n$, where $w_i=\vecone_S\{i\}$ for each $i\in [n]$, is easily seen that $S$ is $\varepsilon$-balanced if and only if $W_S$ is $(|S|/n,\varepsilon)$-uniform.}
  We next show that the properties Exponential sums
  and Equidistribution  are equivalent. Since $f(x)=\exp\left({2\pi i}kx\right)$ integrates to~$0$ and has Lipschitz norm at most $2|k|$, it is clear that the Equidistribution
    property implies the Exponential sums property.
  To show the converse, let $f:\RR/\ZZ\to\CC$ be given.
  We will  show that for any $\eps>0$ and for large $n$, the following holds for $d=\|\bw_n\|_1/n$:

  \[
  \Big|\frac 1n\sum_{j:\bw_n[j]=1}f(j/n)-d\int_{\RR/\ZZ}f\Big|\le \eps\|f\|_{\Lip}.
  \]
  Let $C$ and $c$ be the absolute constants from Lemma~\ref{lem:lip1} and
  Lemma~\ref{lem:lip2}, respectively. 
 Choose~$M$ large enough so that $M/\log M\geq 2C/\eps$ and $n$ large enough so that for all $|m|\leq M$
we have $\left|\sum_{j:\bw_n[j]=1}\exp\left(\tfrac{2\pi i}nmj\right)\right|<\frac{\eps}{2cM}n|m|$.
Applying this bound we obtain
\begin{align*}
  \sum_{j:\bw_n[j]=1}\sigma_Mf(j/n)
  &= \sum_{j:\bw_n[j]=1}\sum_{|m|\le M}\Big(1-\frac{|m|}{M+1}\Big)\widehat{f}(m)\exp\left(\tfrac{2\pi i}nmj\right)\\	
  &= \sum_{|m|\le M}\Big(1-\frac{|m|}{M+1}\Big)\widehat{f}(m)\sum_{j:\bw_n[j]=1}\exp\left(\tfrac{2\pi i}nmj\right)\\
  &= \widehat{f}(0)\cdot dn\pm \frac\eps{2cM} n  \sum_{0<|m|\le M}\left|\Big(1-\frac{|m|}{M+1}\Big)\widehat{f}(m)\right||m|.
\end{align*}
As $\widehat{f}(0)= \int_{\mathbb R/\mathbb Z}f$, we obtain
  from Lemma~\ref{lem:lip2} that  
  \begin{align*}
    \Big|\frac1n\sum_{j:\bw_n[j]=1}\sigma_Mf(j/n) - d\int_{\mathbb R/\mathbb Z}f\Big|
    \leq\frac\eps{2cM}\sum_{0<|m|\le M}\Big|\Big(1-\frac{|m|}{M+1}\Big)\widehat{f}(m)\Big||m| 
    \le\frac \eps2\|f\|_{\Lip}.
  \end{align*}
  By Lemma~\ref{lem:lip1},  triangle inequality
  and the choice of $M$ we conclude
\begin{align*}
\Big|\frac 1n\sum_{j:\bw_n[j]=1}f(j/n)-d\int_{\mathbb R/\mathbb Z}f\Big| 
  & \le \Big|\frac 1n\sum_{j:\bw_n[j]=1}\sigma_Mf(j/n)-d\int_{\mathbb R/\mathbb Z}f\Big |+C\|f\|_{\Lip}\frac{\log M}{M}\\
  & \leq \frac \eps2\|f\|_{\Lip}+\frac\eps2\|f\|_{\Lip}=\eps\|f\|_{\Lip}.
\end{align*}
This finishes the proof.
\end{proof}

\section{Limits of word sequences}\label{sec:limitTheory}
In this section we give the proof of Theorem~\ref{thm:limits}  concerning word limits. Although the overall approach is in line with what has been done for graphons~\cite{LS1} and permutons~\cite{HOPPEN}, there are important technical differences which we will stress below.
Central concepts and auxiliary results involved in the proof will be introduced along the way. 
The section is divided into four subsections.
We start by a simple reformulation of the notion of convergent word sequences in terms of 
 convergence of a function sequence in $\cW$. This notion is called $t$-convergence and we in Lemma~\ref{lem:f=g} show that the limit of a 
 $t$-convergent function sequence is unique, if it exists. In the second subsection, we endow $\cW$ with the interval-distance $d_\Box$
  and show in Lemma~\ref{lem:tcount} that convergence with respect to~$d_\Box$  implies $t$-convergence.
  Proposition~\ref{prop:tconvboxconv} from the same subsection gives a direct proof of the converse.
  In the third subsection, we specify a third and last notion of convergence (convergence in distribution) based on sampling of $f$-random letters for a given $f\in\cW$.
  We prove in Lemma~\ref{lem:boxconvdconv} that this notion of convergence is equivalent to the two previously defined, and deduce
  the compactness of the metric space $(\cW,d_\Box)$ in Theorem~\ref{cor:compact}.
  In the fourth and last part, we show in Lemma~\ref{lem:randomwords} and Corollary~\ref{cor:randomwords} that every element of $f\in\cW$ is, a.s.,
  the limit of a convergent random word sequence.}


{
  
\subsection{Uniqueness and $t$-convergence}\label{sec:tconv}
Given the   nature of the limit it is convenient to  first reformulate the notion of convergence in  analytic terms.
For a given word $\bw_n=(w_1,\dots, w_n)$ define the \emph{function associated to $\bw_n$} to be the $n$-step $0$-$1$-function $f_{\bw_n}\in\cW$
  given
by $f_{\bw_n}(x)=w_{\lceil nx\rceil}$.
It is then easy to see that $t(\bu,f_{\bw_n})$, as defined in~\eqref{eq:subseqdensity}, satisfies\footnote{\label{foot:countcount}To see~\eqref{eq:countcount}, split $[0,1]$ into $n$ intervals of equal lengths. 
Let $A$ denote the event that  $\ell$ independent uniform random  points  of $[0,1]$  land in different intervals
and let $B$ be the event that, after reordering these points, say $x_1<\dots<x_\ell$, we have $\big(f_{\bw_n}(x_1),\dots,f_{\bw_n}(x_\ell)\big)=\bu$.
Then, $t(\bu,f_{\bw_n})=\PP(B|A)\PP(A)+\PP(B|\overline{A})\PP(\overline{A})$ and we further have $\PP(B|A)=t(\bu,\bw_n)$ and $\PP(A)=\prod_{i=1}^{\ell-1}(1-i/n)=1-O(n^{-1})$.}
\begin{align}\label{eq:countcount}t(\bu,f_{\bw_n})=t(\bu,\bw_n)+O(n^{-1}\big)\qquad \text{for every word~$\bu$}.\end{align}
Thus the following, applied to $f_n=f_{\bw_n}$, yields a reformulation  of convergence of  $({\bw_n})_{n\to\infty}$. Given a sequence $(f_n)_{n\to\infty}$ in $\cW$ and $f\in\cW$, we say that 
\[
  {f_n}\overset{t}{\to} f \qquad\text{ if }\qquad  \lim_{n\to\infty}t (\bu,f_{n})=t(\bu,f)\quad \text{ for all finite words $\bu$.}
  \]

The next lemma implies that the limit, if it exists, is guaranteed to be unique. 
The idea of the proof goes back to  a remark of Kr\'al' and Pikhurko
concerning permutons (see~\cite[Remark 6]{KP13}).
\begin{lemma}\label{lem:f=g}
 Let $f,g\in\cW$. If $t(\bu,f)=t(\bu,g)$ for all words $\bu$, then $f=g$ almost everywhere.
\end{lemma}
\begin{proof}Given $k\in\NN$, note that
  \begin{align*}
   \int_{0}^1f(x)x^k\diff x 
    & = \int_0^1f(x)\Big(\int_0^x\diff y\Big)^k\diff x=\int_{y_1,\dots,y_k\le x}f(x)\diff y_1\dots \diff y_k\diff x\\
    & = k!\int_{y_1<\dots<y_k< x}f(x)\diff y_1\dots \diff y_k\diff x=\frac{1}{k+1}\sum_{\bu\in\{0,1\}^k}t(u_1\dots u_k1,f)\\
    & = \frac{1}{k+1}\sum_{\bu\in\{0,1\}^k}t(u_1\dots u_k1,g)=\int_0^1g(x)x^k\diff x.
\end{align*}
Thus, for each polynomial $P(x)\in \mathbb{R}[x]$ we get 
$\int_0^1f(x)P(x)\diff x=\int_0^1g(x)P(x)\diff x,$
and  by the Stone--Weierstrass theorem 
$\int_0^1f(x)h(x)\diff x=\int_0^1g(x)h(x)\diff x$
holds for every continuous function $h:[0,1]\to\mathbb R$. This implies that $f=g$ almost everywhere.
\end{proof}

\subsection{Interval-metric and the metric space $(\cW,d_\Box)$}
In view of  the equivalence of uniformity and
subsequence counts shown in Theorem~\ref{thm:quasirandom}, it is natural to consider the following {notions of norm, distance and convergence, which are all analogues of the notions of cut-norm, cut-distance and convergence in graph limit theory}. Given  $h:[0,1]\to[-1,1]$ define the interval-norm 
\[
\|h\|_{\Box}=\sup_{I\subseteq[0,1]} \left|\int_I h(x)\diff x\right|,
\]
where the supremum is taken over all intervals $I\subseteq[0,1]$.
The interval-metric $d_\Box$ is then defined by
\[
d_{\Box}(f,g)=\|f-g\|_\Box\qquad \text{ for every $f,g:[0,1]\to[0,1]$,}
\]
and we write 
\[
  f_n\overset{\Box}{\to} f \qquad\text{ if }\qquad \lim_{n\to\infty} d_{\Box}(f_{n},f)=0.
  \]
The following result states that the interval-norm
controls subsequence counts, in particular,  $f_n\overset{\Box}{\to} f$ implies $f_n\overset{t}{\to} f$. 
As a byproduct of the lemma, we  obtain the first part of Theorem~\ref{thm:quasirandom} concerning counting subsequences  in uniform words. 

\begin{lemma}\label{lem:tcount}
For $f,g\in\cW$ and $\bu\in\{0,1\}^\ell$ we have
\[
\big|t(\bu,f)-t(\bu,g)\big|\leq\ell^2\cdot d_\Box(f,g).
\]
In particular, if $\bw\in\{0,1\}^n$ is $(d,\eps)$-uniform and $n=n(\eps,\ell)$
is sufficiently large, then for some $d\in [0,1]$
we have for each $\bu\in\{0,1\}^\ell$ 
\[
\tbinom {\bw}\bu=d^{\|\bu\|_1}(1-d)^{\ell-\|\bu\|_1}\tbinom n\ell\pm 5\eps n^\ell.\]
\end{lemma}
\begin{proof}
We first show that the second part follows from the first. Given a $(d,\eps)$-uniform word $\bw\in\{0,1\}^n$, let $f:[0,1]\to[0,1]$ be the function associated to $\bw$.
Define $g:[0,1]\to[0,1]$ to be constant equal to $d$ and recall that $g^1=g$ and $g^0=1-g$. Then, for each $\bu\in\{0,1\}^\ell$
\begin{align*}t(\bu,g)=\ell!\int_{0\leq x_1<\dots<x_\ell\leq1}\prod_{i\in[\ell]}g^{u_i}(x_i)\diff x_1\dots \diff x_\ell=d^{\|\bu\|_1}(1-d)^{\ell-\|\bu\|_1}.\end{align*}
Since $d_\Box(f,g)\leq 2\eps$ due to uniformity of $\bw$, 
for large $n$, the second part of the lemma follows from the first part
and~\eqref{eq:countcount}  as
\[\tbinom \bw\bu=t(\bu,f)\tbinom n\ell\pm \eps n^\ell=t(\bu,g)\tbinom n\ell\pm 5\eps n^\ell
=d^{\|\bu\|_1}(1-d)^{\ell-\|\bu\|_1}\tbinom n\ell\pm 5\eps n^\ell.\]
Now we turn to the proof of the first part.
Let
\[
X_j(x_1,\dots,x_\ell)=\big(f^{u_j}(x_j)-g^{u_j}(x_j)\big)\prod_{i=1}^{j-1}f^{u_i}(x_i)\prod_{i=j+1}^\ell g^{u_i}(x_i).
\]
Making use of a telescoping sum we write
\begin{align*}
\big|t(\bu,f)-t(\bu,g)\big|&= \ell!\Big|\int_{x_1<\dots<x_\ell}\Big(\prod_{j\in[\ell]}f^{u_j}(x_j)-\prod_{j\in[\ell]}g^{u_j}(x_j)\Big) \diff x_1\dots \diff x_\ell\Big|\\
&= \ell!\Big|\int_{x_1<\dots<x_\ell}\sum_{j\in[\ell]}X_j(x_1,\dots,x_\ell)\diff x_1\dots \diff x_\ell\Big|\\
&\leq \ell!\sum_{j\in[\ell]}\Big|\int_{x_1<\dots<x_\ell}X_j(x_1,\dots,x_\ell)\diff x_1\dots \diff x_\ell\Big|.
\end{align*}
Since $\displaystyle\Big|\int_{x_{j-1}}^{x_{j+1}} \big(f^{u_j}(x_j)-g^{u_j}(x_j)\big)\diff x_j\Big|\leq d_\Box(f,g)$ and $0\leq f, g\leq 1$, for $j\in[\ell]$ we have
\[
\Big|\int_{x_{j-1}}^{x_{j+1}}X_{j}(x_1,...,x_\ell)\diff x_j\Big|
  \leq d_{\Box}(f,g)\prod_{i=1}^{j-1}f^{u_i}(x_i)\prod_{i=j+1}^\ell g^{u_i}(x_i){\leq d_{\Box}(f,g)}.
\]
Hence,
\begin{align*}
  & \Big|\int_{x_1<\dots<x_{\ell}} X_j(x_1,\dots,x_{\ell})\diff x_1\dots \diff x_{\ell}\Big| \\
  & \qquad \leq
  d_\Box(f,g)\int_{\substack{x_1<\dots<x_{j-1}\\ \leq x_{j+1}<\dots<x_{\ell}}}
    \diff x_1\dots \diff x_{j-1} \diff x_{j+1}\dots\diff x_{\ell} \\
    & \qquad
    \leq \frac{1}{(\ell-1)!}d_\Box(f,g)
\end{align*}
and the first part of the lemma follows.
\end{proof}

\begin{remark}\label{remark:alphabetcounting}
We note that the same argument extends without change to larger size alphabets
in the following sense.
Given an alphabet $\Sigma=\{a_1,\dots, a_k\}$, let 
$\vecf=(f^{a_1},\dots ,f^{a_k})$ and $\vecg=(g^{a_1},\dots,g^{a_k})$ be two tuples of functions
$f^{a_i},g^{a_i}:[0,1]\to[0,1]$, for $i\in[k]$, such that
\[f^{a_1}(x)+\dots+f^{a_k}(x)=1\mbox{ and } g^{a_1}(x)+\dots+g^{a_k}(x)=1 \mbox{ almost everywhere.}\]
For a word $\bu\in\Sigma^\ell$, define the density of $\bu$ in $\vecf$ in
similar manner as in~\eqref{eq:subseqdensity}, namely
\[
t(\bu,\vecf)=\ell!\int_{0\leq x_1<\dots<x_\ell\leq 1}\prod_{i\in[k]}f^{u_i}(x_i)\diff x_1\dots \diff x_\ell.
\]
Then, the  proof from above yields
\[
\big|t(\bu,\vecf)-t(\bu,\vecg)\big|\leq\ell^2\cdot\max_{i\in[k]}d_{\Box}(f^{a_i},g^{a_i}).
\]
\end{remark}

Note that Lemma~\ref{lem:tcount} implies
that if $f_n\overset{\Box}{\to} f$, then $f_n\overset{t}{\to} f$.
Our goal now is to show that the converse also holds.
Let $(f_n)_{n\to\infty}$ be a sequence such that $f_n\overset{t}{\to} f$.
Following the proof of Lemma~\ref{lem:f=g}, we will use that for any
polynomial $P(x)\in\mathbb R[x]$ we can write $\int_0^1(f_n(x)-f(x))P(x)$
as a linear combination of subsequence densities.
By approximating $\mathbf 1_{[a,b]}(x)$ by a polynomial $P_{a,b}(x)\in \mathbb R[x]$, with error term uniform in $0\le a<b\le 1$, we may show that
$\int_0^1(f_n(x)-f(x))\mathbf1_{[a,b]}(x)$ can be approximated by
$\int_0^1(f_n(x)-f(x))P_{a,b}(x)$, thence by
a linear combination of subsequence densities,  implying our claim. In order to prove this approximation result, we introduce next the class of Bernstein polynomials.

Even though here we only need to approximate functions on $[0,1]$, we will consider the general case of functions on $[0,1]^k$ since it will later be useful
in our study of higher dimensional combinatorial structures. For $k,t\in\mathbb N\setminus\{0\}$, let $\bi=(i_1,\dots i_k)\in[t]^k$. Given a function $J:[0,1]^k\to\RR$, define its Bernstein polynomial evaluated at $\bx=(x_1,\dots,x_k)\in [0,1]^k$ by
\[B_{t, J}(\bx)=\sum_{0\leq i_1,\dots,i_k\leq t}J(\tfrac{\bi}{t})\prod_{j\in[k]}{\tbinom{t}{i_j}x_j^{i_j}(1-x_{j})^{t-i_j}}.\]
We can now formally state the approximation of indicator functions
we use.
\begin{lemma}\label{lem:Bta2}
For $\ba=(a_1,\dots, a_k)\in[0,1]^k$ let $J=\vecone_{[0,a_1]\times\dots\times [0,a_k]}$.
If $r\in\NN$ and $\bx\in[0,1]^k$ satisfy $|x_i-a_i|>r^{-1/4}$ for all $i\in[k]$, then
	$|B_{r,J}(\bx)-J(\bx)|\leq k r^{-1/2}.$
\end{lemma}
\begin{proof}
  Let $B=B_{r,J}$
  and let $\bX=(X_1,...,X_r)$ be such that $X_1,...,X_r$ are independent random
  variables where $X_j$ follows a binomial distribution with parameters $r$ and $x_j$, so
      $\PP(X_j=i)=\binom{r}
      {i}x_j^i(1-x_j)^{r-i}$.
  Note that $B_{r,J}(\bx) = \EE(J(\tfrac1r\bX))$.
  Let $L=\{\bi \colon \|\bx-\tfrac{\bi}r\|_{\infty}> r^{-1/4}\}\subseteq (\{0\}\cup [r])^k$.
  As  $|x_j-a_j|>r^{-1/4}$ for all $j\in[k]$, for each $\bi \not\in L$ we have
  that $J(\tfrac {\bi}r)=J(\bx )$ and thus
  \[
  \EE\big(\big|J(\tfrac{1}{r}\bX)-J(\bx)\big|\mathbf{1}_{{\overline{L}}}(\bX)\big)
  = 0.
  \]
  Due to $\big|J(\tfrac {\bi}r)-J(\bx)\big|\leq 1$  we have 
  \begin{align}\label{eq:Lk}
    \EE\big(|J(\tfrac1r\bX)-J(\bx)|\big)
    = \EE\big(|J(\tfrac1r\bX)-J(\bx)|\mathbf{1}_L(\bX)\big)
    \leq \PP(\bX\in L)
    \leq \sum_{\ell\in [k]}\PP(|\tfrac{1}{r}X_\ell-x_\ell|>r^{-1/4}).
  \end{align}
  Since $\EE(X_{\ell})=rx_\ell$, by Chebyshev's inequality, 
  \[
  \PP(|\tfrac{1}{r}X_\ell-x_\ell|>r^{-1/4})=\PP(|X_{\ell}-\EE(X_\ell)|>r^{3/4})\leq \frac{1}{r^{3/2}}rx_\ell(1-x_\ell)
  \leq r^{-1/2}
  \]
  Since the bound holds for every $\ell\in [k]$,
  the RHS of~\eqref{eq:Lk} is at most $kr^{-1/2}$, as required.
\end{proof}
Given two functions $f,g\in\cW$, we have the inequality
   \begin{equation}\label{bound:distribution}\sup_{b\in[0,1]}\left|\int_0^bf(x)\diff x-\int_0^bg(x)\diff x\right|\le d_\Box(f,g)\le 2 \sup_{b\in[0,1]}\left|\int_0^bf(x)\diff x-\int_0^bg(x)\diff x\right|.  \end{equation}
The first inequality in~\eqref{bound:distribution} is direct from the  definition of $d_\Box$, and the second inequality follows from the identity  $\int_0^b(f(x)-g(x))=\int_a^b(f(x)-g(x))+\int_0^a(f(x)-g(x))$.

The following proposition states that $t$-convergence implies convergence with respect to $d_\Box$, and thus, together with Lemma~\ref{lem:tcount}, establishes that both notions of convergence are equivalent.
\begin{proposition}\label{prop:tconvboxconv}
  If $(f_n)_{n\to\infty}$ is a sequence in $\cW$ which is $t$-convergent,
  then it is a Cauchy sequence with respect to $d_{\Box}$.
  Moreover, if  $f_n\stackrel{t}\to f$ for some $f\in\mathcal W$, then $f_n\stackrel{\Box}\to f.$
\end{proposition}
\begin{proof}
  Given $\eps>0$, let $r=\lceil(20/\eps)^4\rceil$.
  For $\delta=\eps/2^{3r+2}$, let $n_0$ be sufficiently large so that for all $n,m\geq n_0$ we have 
  \begin{align}\label{eq:cauchy}\big|t(\bu,f_n)-t(\bu,f_m)\big|\leq \delta\quad \text{ for all } \bu\in\bigcup_{s\in [r]}\{0,1\}^{s}.\end{align}
	Recall from the proof of Lemma~\ref{lem:f=g}, that for each $k\in\mathbb N$ we have
\[\int_{0}^1f(x)x^k\diff x=\frac{1}{k+1}\sum_{\bu\in\{0,1\}^k}t(u_1\dots u_k1,f).\]
	Thus, for $k\leq r$ and $h=f_n-f_m$, we have
	\begin{align*}
          \Big|\int_{0}^1h(x)x^k\diff x\Big|
          =\frac{1}{k{+}1}\Big|\sum_{\bu\in\{0,1\}^k}(t(u_1\dots u_k1,f_n)-t(u_1,\dots, u_k1,f_m))\Big|\leq {2^k\delta}.            
        \end{align*}
	For $a\in [0,1]$, let
        $J_a=\vecone_{[0,a]}$ and $j_a$ be the largest {integer} such that
        $\tfrac{j_a}r\leq a$. Then,
	\[
	\left|\int_0^1 h(x)B_{r,J_a}(x)\diff x\right|\le\sum_{i=0}^{j_a}\tbinom ri\left|\int_0^1  h(x)x^i(1-x)^{r-i}\diff x\right|
        \leq 2^{3r}\delta.
	\]
        Thus, since $|h|\leq 1$ and $|\vecone_{[0,a]}(x)-B_{r,J_a}|\leq 2$, by Lemma~\ref{lem:Bta2}, we have

\begin{align*} \left|\int_0^1 h(x)\vecone_{[0,a]}(x)\diff x\right|
	&\le
	 \left|\int_0^1 h(x)B_{r,J_a}(x)\diff x\right| + \left|\int_0^1 h(x)(\vecone_{[0,a]}(x)-B_{r,J_a}(x))\diff x\right|\\
	 \\\nonumber
	&\leq   2^{3r}\delta + (4r^{-1/4}+r^{-1/2}).
\end{align*}
The desired conclusion follows from~\eqref{bound:distribution} and by our choice of $t$ and $\delta$ observing
  that 
\[
d_{\Box}(f_n,f_m)
\leq 2\sup_{a\in[0,1]}\left|\int_0^1h(x)\vecone_{[0,a]}(x)\diff x\right| 
\leq 2^{3r+1}\delta+10r^{-{1/4}}\leq\eps.
\]
The second part follows by replacing $f_m$ by $f$ in~\eqref{eq:cauchy},
  taking $h=f_n-f$, and repeating the above argument. 
\end{proof}

The compactness of the metric space $(\cW,d_\Box)$ can be easily
established via the Banach--Alaoglu theorem in $L^{\infty}([0,1])$.
Instead, we follow a different strategy laid out in the
following section. This strategy  has the advantage that it emphasizes the
probabilistic point of view of convergence. It is based on a new model of random
words that naturally arises from the theory and that may be of independent interest.    

We note that one can also establish the
compactness of $(\cW,d_\Box)$ by using the regularity lemma for
words~\cite{APP13}.  This approach has the advantage of being more
constructive and for the sake of completeness we include it in the 
Appendix~\ref{sec:appendix}.

\subsection{Random letters from limits and compactness of $(\cW,d_\Box)$} 
Consider  the {Euclidean} metric on $[0,1]$ and the discrete metric on $\{0,1\}$. 
Let  $\Omega=[0,1]\times\{0,1\}$ be equipped with the $L_\infty$-distance, which  thus assigns to a pair of points in $\Omega$ the standard distance of their  first coordinates 
if the second coordinates agree and  one otherwise. Let $\cB$ denote the Borel $\sigma$-algebra of $\Omega$, let
  $f:[0,1]\to[0,1]$ be  a Borel measurable function and
recall that $f^1=f$ and $f^0=1-f$. 
{Also, denote by $\rU([0,1])$ and $\rB(p)$ the uniform distribution over $[0,1]$ and the Bernoulli distribution with expected value $p\in [0,1]$, respectively. We say that
  \[
  \text{$(X,Y)\in\Omega$ is an $f$-random letter}
  \qquad\text{ if }\qquad
  \text{$X\sim\rU([0,1])$ and $Y\sim\rB(f(X))$.}
  \]}
Observe that an $f$-random letter $(X,Y)$ is a pair of mixed\footnote{Mixed in the sense that $X$ is continuous while $Y$ is discrete.} random variables where  $Y$ is distributed according to the conditional pmf
\[
f_{Y|X}(\eps|x)
  =\PP(Y=\eps|X=x)=f^\eps(x)\qquad\eps\in\{0,1\} \text{ and } x\in[0,1].
\]
Then, $(X,Y)$ has the mixed joint probability distribution 
\begin{align}\label{eq:jointXY}
  F(x,\eps)=\PP(X\leq x, Y=\eps)=\int_0^xf^\eps(t)\diff t,
\end{align}
and thus the mixed joint pmf $f_{X,Y}(x,\eps)=f^\eps(x)$. The marginal probability distribution of $Y$ is
 \[\PP(Y=\eps)=F(1,\eps)=\int_0^1f^\eps(t)\diff t,\qquad \eps\in\{0,1\},\]
 hence  $Y\sim\rB(p)$ with $p=\int_0^1f(t)\diff t$.
Furthermore, conditioned on $Y$ the variable~$X$  is distributed according to the conditional pmf $f_{X|Y}$ which satisfies 
\begin{align}\label{eq:fXY}f_{X|Y}(x|\eps)\cdot\PP(Y=\eps)=f_{X,Y}(x,\eps)={f^\eps(x)}.\end{align}
One may therefore equivalently sample $(X,Y)$ by first choosing  $Y\sim\rB(p)$ with  $p=\int_0^1f(t)\diff t$, and
then choose $X$ (conditional on $Y$)  according to the conditional pmf $f_{X|Y}$ satisfying~\eqref{eq:fXY}.
By means of this sampling procedure a sequence
$(f_n)_{n\to\infty}$ gives rise to a sequence $\big((X_n,Y_n)\big)_{n\to\infty}$, where each $(X_n,Y_n)$ is the $f_n$-random letter,
 and the corresponding sequence of probability distributions $(\PP_n)_{n\to\infty}$ is as defined in~\eqref{eq:jointXY}.
 As usual for general metric spaces (see, e.g.,~\cite[Chapter 5]{billing}), 
 we say that $\big((X_n,Y_n)\big)_{n\to\infty}$ converges to $(X,Y)$ in distribution if  $(\PP_n)_{n\to\infty}$ weakly converges to $\PP$, i.e.,
 if for all bounded continuous functions $h:\Omega\to \mathbb R$ we have
\begin{align}\label{eq:defdconv}\lim_{n\to\infty}\int_\Omega h\diff\PP_n=\int_\Omega h\diff\PP.\end{align}
From this definition we immediately have the following. 
\begin{fact}\label{fact:jointmarginal}
If $\big((X_n,Y_n)\big)_{n\to\infty}$ converges to $(X,Y)$ in distribution, then $(X_n)_{n\to\infty}$ (resp., $(Y_n)_{n\to\infty}$) converges to $X$ (resp. $Y$) in
distribution. 
\end{fact} We now  write 
\[
  {f_n}\overset{\mathrm d}{\to} f \qquad\text{ if }\qquad  \text{$\big((X_n,Y_n)\big)_{n\to\infty}$ {converges to $(X,Y)$ in distribution}}.
\]
The next lemma shows the equivalences of convergence in $d_\Box$ and
convergence in distribution.
\begin{lemma}\label{lem:boxconvdconv}
  Let $f_1,f_2,\dots$ and $f$ be functions in $\cW$.
  Then, $f_n\overset{\Box}{\to} f$ if and only if $f_n\overset{\mathrm{d}}{\to} f$.
\end{lemma}
\begin{proof}
  Let $(X_n,Y_n)$ be an $f_{n}$-random letter (resp. $(X,Y)$ be an $f$-random
  letter)  with the associated probability measure
  $\PP_n$ and distribution $F_n$ (resp. $\PP$
  and $F$).
  Recall that $\Omega=[0,1]\times \{0,1\}$, {and  for $(x,\eps)\in \Omega$ let $F_n(x,\eps)=\int_0^xf_n^{\eps}(t)\diff t$ and $F(x,\eps)=\int_0^xf^{\eps}(t)\diff t$.}
  Since,
  \[ \|F_n-F\|_\infty=\sup_{(x,\eps)\in\Omega} |F_n(x,\eps)-F(x,\eps)|\] 
  it follows that
  \[\|F_n-F\|_{\infty}=\sup_{x\in[0,1]}|F_n(x,0)-F(x,0)|=\sup_{x\in[0,1]}|F_n(x,1)-F(x,1)|=\sup_{x\in [0,1]}\Big|\int_{0}^x(f_n-f)(t)\diff t\Big|.\]
 Now observe that
    \begin{equation}\label{eq:Fn}\|F_n-F\|_\infty\leq d_\Box(f_n,f)\leq 2\|F_n-F\|_\infty,\end{equation}
   {where the first inequality is obvious and the second one 
      follows because for all $\eps\in\{0,1\}$ and
      $0\leq a<b\leq 1$ it holds that      $\int_{[a,b]}(f_n-f)(t)\diff t=(F_n-F)(b,\eps)-(F_n-F)(a,\eps)$.}
  Thus, $f_n\overset{\Box}{\to} f$ if and only if
  $\lim_{n\to\infty}\|F_n-F\|_\infty=0$ which we claim holds if and only if
  \begin{equation}\label{eq:Fepsinf}
  \lim_{n\to\infty}F_{n}(x,\eps) = F(x,\eps)\qquad  \text{for all } \eps\in\{0,1\}\text{ and } x\in [0,1].
  \end{equation}
  Indeed, it is clear that $\lim_{n\to\infty}\|F_n-F\|_\infty=0$
  implies~\eqref{eq:Fepsinf}. 
  For the converse note that for each $\eps\in\{0,1\}$ we have
  $|f^\eps|\leq1$, thus for every $x,y\in[0,1]$
  \begin{equation}\label{eq:Fx-Fy}
    |F(x,\eps)-F(y,\eps)|
    =\left|\int_0^xf^\eps(t)\diff t-\int_0^yf^\eps(t)\diff t\right|\leq  |x-y|.
  \end{equation}
  Given an integer $k>0$, by~\eqref{eq:Fepsinf}, there is an $n_k$  such that
  $\max_{i\in [k]}\left|F_{n}\left(\frac ik,\eps\right)-F\left(\frac ik,\eps\right)\right|<\frac 1k$   for each $n>n_k$.
For an $x\in[0,1]$ let $i_x\in[k]$ be such that $|x- \frac{i_x}k|\leq\frac 1k$. 
Then, by triangle inequality and~\eqref{eq:Fx-Fy}, for any $x\in[0,1]$
\begin{align*}
\textstyle\left|F_{n}\left(x,\eps\right)-F\left(x,\eps\right)\right|\leq \left|F_{n}\left( \tfrac {i_x}k,\eps\right)-F\left(\tfrac {i_x}k,\eps\right)\right|+2|x-\tfrac {i_x}k|\leq \tfrac 3k
\end{align*}
which thus establishes  that~\eqref{eq:Fepsinf} implies $\lim_{n\to\infty}\|F_n-F\|_\infty=0$.

To prove the lemma we now show  that~\eqref{eq:Fepsinf} holds  if and only if $(X_1,Y_1)$, $(X_2,Y_2),\dots$ converges to $(X,Y)$ in distribution, i.e.,
$\PP_1, \PP_2,\dots$ weakly converges to~$\PP$ as defined in~\eqref{eq:defdconv}.
For an  $h:\Omega\to \RR$ and an $\eps\in\{0,1\}$ define the projection $h_\eps: [0,1]\to\RR$ via $h_\eps(x)=h(x,\eps)$. 
Thus, $F_\eps(x)=F(x,\eps)$, $F_{n,\eps}(x)=F_n(x,\eps)$ and we also define $\PP_\eps$ via
$\PP_\eps(A)=\PP(A\times\{\eps\})$ for any $A\in\cB([0,1])$ and in the same manner define~$\PP_{n,\eps}$.

For a metric space $(M,d)$, we denote by $C(M)$ the set of continuous functions $h:M\to\RR$.
As $\Omega$ is equipped with $L_\infty$-distance $d_\Omega$ we have 
$d_\Omega((x,\alpha),(y,\beta))= \delta<1$  if an only if $\alpha=\beta$ and $|x-y|=\delta$. Hence,  $h\in C(\Omega)$ 
if and only if $h_0,h_1\in C([0,1])$.
Moreover, by verifying the following for step functions~$h$ and then extending to all $h\in C(\Omega)$ by a standard limiting argument we have
 \[\int_\Omega h\diff\PP_n=\sum_\eps\int_{[0,1]}h_\eps\diff\PP_{n,\eps}\qquad\text{and }\qquad \int_\Omega h\diff\PP=\sum_\eps\int_{[0,1]}h_\eps\diff\PP_\eps.\] 
In particular,  
\[\lim_{n\to\infty}\int_\Omega h\diff\PP_n=\int_\Omega h\diff\PP\quad\text{for all $h\in C(\Omega)$}\] 
holds if and only if 
\[\lim_{n\to\infty}\int_\Omega h\diff\PP_{n,\eps}=\int_\Omega h\diff\PP_\eps \quad \text{for all $\eps\in\{0,1\}$, and all $h\in C([0,1])$}.\]
 In other words, $\PP_1,\PP_2,\dots$  converges weakly to $\PP$ if and only if  $\PP_{1,\eps},\PP_{2,\eps},\dots$  converges weakly to $\PP_\eps$ for all $\eps\in\{0,1\}$.
 As the underlying space is $[0,1]$ it is 
 well known that weak convergence of $\PP_{1,\eps},\PP_{2,\eps},\dots$  to
 $\PP_\eps$ is equivalent to the fact  that $\lim_{n\to\infty}F_{n,\eps}(x) = F_\eps(x)$ holds for all $x$ where $F_\eps(x)$ is continuous. As seen from~\eqref{eq:Fx-Fy}, $F_\eps$ is continuous on the entirety of  $[0,1]$. 
This thus shows that weak convergence of $\PP_1,\PP_2,\dots$ to $\PP$ is equivalent to~\eqref{eq:Fepsinf} and the lemma follows.
\end{proof}

 The compactness of  $(\cW,d_\Box)$ now follows
 from  Lemma~\ref{lem:boxconvdconv}  and classical results from measure theory, namely Prokhorov's theorem concerning the existence of weak convergent subsequences for a given 
sequence of measures over compact measurable spaces and Radon--Nikodym theorem concerning the existence of derivatives of  measures  which are absolutely continuous
 with respect to the Lebesgue measure.
\begin{theorem}\label{cor:compact}
The metric  space $(\cW,d_{\Box})$ is compact.
\end{theorem}
\begin{proof}
Given a sequence $(f_n)_{n\to\infty}$ of functions $f_n\in\cW$. Consider the sequence of $f_n$-random letters $\big((X_n,Y_n)\big)_{n\to\infty}$
with the corresponding sequence of probabilities $(\PP_n)_{n\to\infty}$ on
$(\Omega,\cB)$  defined by~\eqref{eq:jointXY}. As~$\Omega$ is compact we conclude from Prokhorov's theorem (see Chapter~1, Section 5 of~\cite{billing})
 that there is a pair of random variables $(X,Y)$  
 with joint probability measure $\PP$ such that $(\PP_n)_{n\to\infty}$ contains a  subsequence $(\PP_{n_i})_{i\to\infty}$ which  weakly converges to~$\PP$. 
By Fact~\ref{fact:jointmarginal} we know that $X\sim\mathrm{U}[0,1]$ while $Y$ is Bernoulli. Denoting by $\lambda$ the Lebesgue measure, the restriction of $\PP$ to $Y=1$ 
yields a measure $\mu$ which satisfies  
$\mu(A)=\PP(X\in A,Y=1)\leq {\PP(X\in A)=} \lambda(A)$ for every measurable set $A$. In particular, $\mu$  is absolutely continuous with respect to the Lebesgue measure $\lambda$ (i.e.,  $\mu(A)=0$ whenever 
$\lambda(A)=0$) and
the Radon--Nikodym theorem guarantees the existence of a function $f$ such that \[\mu([0,x])=\int_0^x f(t)\diff t=\PP(X\leq x, Y=1)\] and
thus
\[\PP(X\leq x, Y=0)=x-\mu([0,x])=\int_0^x (1-f(t))\diff t.\]
In other words, $f_{X,Y}(x,\eps)=f^\eps(x)$ is the pmf of $(X,Y)$ and we thus have $f_{n_i}\overset{\mathrm{d}}{\to} f$. Lemma~\ref{lem:boxconvdconv} guarantees that $f_{n_i}\overset{\Box}{\to} f$ as well.
Lastly, it is easily seen that $f(x)\in[0,1]$ almost everywhere and we may therefore assume that $f\in\cW$.
\end{proof}
The last theorem thus establishes the existence of the limit object claimed
in the first part of Theorem~\ref{thm:limits}. 

\subsection{Random words from limits}
To establish the second part of 
Theorem~\ref{thm:limits} we consider, for any
$f\in\cW$, a suitable sequence of random words arising from $f$ and show that it converges to $f$ almost surely. 
For $f\in\cW$ and $\bx=(x_1,...,x_{\ell})\in [0,1]^{\ell}$
  such that $x_1<x_2<...<x_\ell$
  let $\bw=\sub(\bx,f)$ be the word obtained by choosing
  $w_i=1$ with probability $f(x_i)$ and $w_i=0$ with probability
  $1-f(x_i)$ (making independent decisions for different $i$'s).
Consider now $n$ independent $f$-random
letters $(X_1,Y_1), \dots, (X_n,Y_n)$.
After reordering the first coordinate, i.e., taking a permutation
$\sigma:[n]\to[n]$ so that $X_{\sigma(1)}<\dots< X_{\sigma(n)}$,
the \emph{$f$-random word} $\sub(n,f)$ is given by
\[
\sub(n,f)=(Y_{\sigma(1)},\dots, Y_{\sigma(n)}).
\]

\begin{lemma}\label{lem:randomwords}
  Let  $f\in \cW$ and let $f_n$ be the function associated to the $f$-random
  word $\sub(n,f)$.
  For all $n\in\mathbb N$ and $a\ge\frac 1n$ we have
  \[
  \PP\big(d_\Box(f_n,f)\ge{8a}\big)\le 4ne^{-2a^2n}.
  \]
\end{lemma}
\begin{proof}

For $x\in [0,1]$ let
  \[
  W_n(x)=\int_0^xf_n(t)\diff t\hspace{.5cm}\mbox{and}\hspace{.5cm}W(x)=\int_0^xf(t)\diff t.
  \]
  Recall that by~\eqref{bound:distribution} we have 
 $d_\Box(f_n,f)\le 2\|W_n-W\|_\infty$. Therefore, we only need to bound $\PP(\|W_n-W\|_{\infty}\ge 5a)$. 
 
 Given $i\in[n]$ and  $x\in[\tfrac {i-1}n,\tfrac{i}{n})$, {since $|f_n|, |f|\leq 1$}, we have that $|W_n(x)-W(x)|\le |W_n(\tfrac{i}n)-W(\tfrac{i}{n})|+\frac 2n$, and thus 
\[\|W_n-W\|_\infty \le \frac 2n+\max_{i\in[n]}|W_n(\tfrac{i}n)-W(\tfrac{i}{n})|. \]
For $i\in[n]$, we next bound the probability that $|W_n(\tfrac in)-W(\tfrac in)|$ is at least {$2a$}. Consider the sequence $(X_1,Y_1),\dots,(X_n,Y_n)$ of $f$-random letters that define $\sub(n,f)$, and suppose that $X_{\sigma(1)}<\dots< X_{\sigma(n)}$ for some permutation $\sigma:[n]\to[n]$. Since $f_n$ is the function associated to $\sub(n,f)$ we have
\[W_n(\tfrac in)=\frac1n\sum_{j=1}^iY_{\sigma(j)}
\]
and thus, letting $Z_i=\frac 1n\sum_{j=1}^n\vecone\{X_j\le \tfrac in\}$ and
$S_i=\frac1n\sum_{j=1}^n{Y_j}\vecone\{X_j\le \tfrac in\}{=\frac1n\sum_{j=1}^{{n}Z_i} Y_{\sigma(j)}}$, we get
\begin{equation}\label{eq:onesuptoi}
\Big|W_n(\tfrac in)-S_i\Big|\le \Big|\frac in-Z_i\Big|.
\end{equation}
On the other hand, for every $j\in[n]$ we have that
\[
\EE(Y_j\vecone\{X_j\le \tfrac in\})=\int_0^{\frac{i}{n}}f(t)\diff t=W(\tfrac in),
\]
so $\EE(S_i)=W(\frac in)$.
Using Chernoff's bound (see Theorem~2.8 and Remark 2.5 from~\cite{Randomgraphs}) we get
\[
\PP(\big|Z_i-\tfrac in\big|\ge a)\le 2e^{-2a^2n}\hspace{.5cm}\text{and}\hspace{.5cm}\PP(\big|S_i-W(\tfrac in)\big|\ge a)\le 2e^{-2a^2n},
\]
which together with~\eqref{eq:onesuptoi} and the fact that $a\geq \frac{1}{n}$, implies that
\[
\PP(|W_n(\tfrac in)-W(\tfrac in)|\ge 2a)\le \PP(|S_i-W(\tfrac in))|\ge a)+
\PP(\big|Z_i-\tfrac in\big|\ge a)
\le 4e^{-2a^2n}.
\]
Putting everything together we conclude that 
\[
  \PP(d_\Box(f_n,f)\ge 8a) \le \PP(\|W_n-W\|_\infty\ge 4a)
	\le \sum_{i=1}^n\PP(|W_n(\tfrac in)-W(\tfrac in)|\ge 2a)\\
	\le 4ne^{-2a^2n}.
\]
\end{proof}
As an immediate consequence we obtain the following.  
\begin{corollary}\label{cor:randomwords}
  For all $f\in\cW$, the sequence of $f$-random words
  $(\sub(n,f))_{n\to\infty}$ converges to $f$ a.s.
\end{corollary}
\begin{proof}
For $n\in\NN$ let $f_n=\sub(n,f)$.
Taking $a=n^{-\frac 14}$ in Lemma~\ref{lem:randomwords} and using the Borel--Cantelli lemma, it follows that $f_n\overset{\Box}{\to}f$ almost surely.
Then, by Lemma~\ref{lem:tcount} we conclude that
  $f_n\overset{t}{\to}f$ almost surely, and therefore,
  by~\eqref{eq:countcount}, $(\sub(n,f))_{n\to\infty}$ converges to $f$ almost
  surely.
\end{proof}

Equipped with the results from above we now establish the second main result of this section.
\begin{proof}[Proof (of Theorem~\ref{thm:limits})]
The uniqueness of the limit, if it exists, follows from Lemma~\ref{lem:f=g}. The second part of the theorem concerning the existence of  word sequences  converging to any given $f\in\cW$ 
follows from Corollary~\ref{cor:randomwords}.

It is thus left to  establish the existence of a limit. 
Consider a convergent sequence $(\bw_n)_{n\to\infty}$ of words and let $(f_n)_{n\to\infty}$ be the sequence of associated functions $f_n=f_{\bw_n}\in\cW$. Because of~\eqref{eq:countcount} the sequence $(f_n)_{n\to\infty}$ is $t$-convergent and thus,  by Proposition~\ref{prop:tconvboxconv}, $(f_n)_{n\to\infty}$ is a Cauchy sequence with respect to $d_\Box$. The compactness of $(\cW,d_\Box)$, as guaranteed by Theorem~\ref{cor:compact}, implies that
there exists $f\in\cW$ such that $d_\Box(f_n,f)\to 0$.
Finally, because of Lemma~\ref{lem:tcount} we have that $f_{n}\overset{t}{\to}f$ and therefore $(\bw_n)_{n\to\infty}$ converges to~$f$.
\end{proof}
Concluding this section, and in preparation for the next one, we show that a tail bound on $d_{\Box}(f_{\bu},f_{\bw})$ similar to the one
of Lemma~\ref{lem:randomwords}
holds if instead of sampling an $f_{\bw}$-random word for some
word $\bw$, we sample a subsequence $\bu=\sub(\ell,\bw)$.

  \begin{lemma}\label{lem:combRandomWords}
   For all $\eps>0$ there is an $\ell_0$ such that for any $\bw\in\{0,1\}^n$ and $n\geq \ell\geq \ell_0$    
   the random word $\bu=\sub(\ell,\bw)$ satisfies
    \[
    \PP\big(d_{\Box}(f_{\bu},f_{\bw})\geq \eps\big) \leq 4\ell e^{-\eps^2\ell/300}.
    \]
\end{lemma}
  \begin{proof}
  For given $\eps>0$ we choose $\ell_0=\ell_0(\eps)$ to be sufficiently large and let  $n\geq \ell>\ell_0$. Let $J\in\binom{[n]}{\ell}$ chosen uniformly at random
  and define $\bu=\sub(J,\bw)=u_1\dots u_\ell$. Consider the intervals  $I_i=\{1,\dots,\lfloor i\cdot \frac n\ell\rfloor\}$, $i\in[\ell]$. 
  By definition of cut-distance and that of $f_{\bu}$ and $f_{\bw}$ we have
    \begin{align*}d_\Box(f_{\bu},f_{\bw})&\le 2\sup_{x\in [0,1]}\Big|\int_0^{x}(f_{\bu}-f_{\bw})(t)\diff t\Big|
    \le 2\max_{i\in[\ell]}\Big |\int_0^{\frac {|I_i|}{n}}(f_{\bu}-f_{\bw})(t)\diff t\Big|
    +\frac 2\ell\\
    &\leq 2\max_{i\in[\ell]}  \Big|\frac1\ell\sum_{j=1}^{i}u_j-\frac1n\sum_{j=1}^{|I_i|}w_j\Big |+\frac 4\ell.
    \end{align*}
Thus the lemma follows once we have shown that
\begin{align}\label{eq:ZU}\PP\Big(\Big|\sum_{j=1}^{i}u_j-\frac\ell n\sum_{j=1}^{|I_i|}w_j\Big |>\frac{\eps \ell}3 \Big)\leq 4\ell e^{-{\eps^2\ell}/{300}}\quad \text{for all $i\in[\ell]$}.\end{align}
To show~\eqref{eq:ZU} consider for an $i\in[\ell]$ \[Z_i=|J\cap I_i|\qquad\text{and}\qquad U_i=\sum_{j\in I_i}w_j\vecone\{j\in J\}=\sum_{j=1}^{Z_i} u_j.\]
  Both are hypergeometric random variables with expectations 
  \[\EE(Z_i)=\frac{\ell}n |I_i|=i\pm1 \qquad\text{and}\qquad \EE(U_i)=\frac \ell n\sum_{j=1}^{|I_i|}w_j.\]
    Moreover, $X\in\{Z_i,U_i\}$  satisfies the concentration bound  (see (2.5), (2.6) and Theorem 2.10 from~\cite{Randomgraphs})
    \begin{align*}
      \mathbb P(|X-\EE(X)|\geq t)\le 2\exp\left(-\frac{t^2}{2 (\mathbb E(X)+t/3)}\right)\qquad t\geq 0.
    \end{align*}    
Thus, for $t=\eps \ell/10$ we conclude  that with probability at least $1-4\exp(-\eps^2\ell/300)$ we have
\begin{align*}
\Big|Z_i-i \Big|\leq \frac{\eps\ell}{10}+ 1 \qquad\text{and}\qquad \Big|U_i-\frac \ell n\sum_{j=1}^{|I_i|}w_j\Big|\leq \frac{\eps\ell}{10}\qquad \text{for all $i\in[\ell]$}.
\end{align*}
Finally, for a choice $J\in\binom{[n]}\ell$ which satisfies both of these properties we have for large  $\ell_0$ 
\begin{align*}
\sum_{j=1}^{i}u_j=\sum_{j=1}^{Z_i}u_j\pm\frac{\eps \ell}9=\frac\ell n\sum_{j=1}^{|I_i|}w_j\pm  \frac{\eps \ell}3.
\end{align*}
This proves~\eqref{eq:ZU} and the lemma follows.
\end{proof}

\section{Testing hereditary word properties}\label{sec:testing}
We now turn our focus to  algorithmic considerations,
specifically, to the study of testable word properties via word limits. 
The presentation below is heavily influenced by the derivation of
analogous results for graphons by Lov\'asz and Szegedy~\cite{LStesting}
(for related results concerning testability of permutation properties and
limit objects see~\cite{jkmm11,kk14}).

Let us briefly discuss the approach and the organization of this section.
While the distance of choice for word limit is the interval-metric $d_{\Box}$, property testing inherently deals
with~the normalized Hamming metric (between  words $\bw\in\{0,1\}^n$ and a word property $\cP$) which we recall to be
\begin{align}\label{eq:d1} d_1(\bw,\cP)=\min_{\bu\in\cP\cap\{0,1\}^n}d_1(\bw,\bu)\qquad\text{where}\qquad   d_1(\bw,\bu)=\frac 1n\sum_{i\in [n]}|w_i-u_i|.\end{align}
Being able to relate these two metrics by means of the limit theory is the essence of our approach and formally this is done via the notion of
\emph{closure}~$\overline{\cP}$ of a word property $\cP$, defined as
\[
  \overline{\cP} =\{ f\in\cW \colon\,{\text{There is a sequence $(\bw_n)_{n\to\infty}$ in $\cP$ which converges to $f$}} \}.
\]
Note that $\overline\cP$ may not contain $f_{\bw}$ for a $\bw\in\cP$. However,  compactness of $(\cW,d_\Box)$  immediately implies that 
$f_{\bw}$ is close to~$\overline\cP$ in the interval-metric
if  $\bw\in\cP$ is large enough (see Lemma~\ref{lem:inclosure}). 
Alternative characterizations of $\overline\cP$ for a hereditary word property $\cP$ will be given in Proposition~\ref{prop:threeNine} which moreover shows
that non-trivial hereditary $\cP$ admits only  $0$-$1$ valued $f\in\overline\cP$ (up to a null measure set).

The  $d_1$-metric in~\eqref{eq:d1} can be analogously defined for functions $f$ and $\overline\cP$ as follows (as usual  let $d_1(f,\overline\cP)=\infty$ if $\overline\cP=\emptyset$):
 \[d_1(f,\overline{\cP})=\inf_{g\in\overline{\cP}}d_1(f,g)\qquad\text{where}\qquad
 d_1(f,g)=\|f-g\|_1=\int |f(x)-g(x)|\diff x.\]
With respect to the closure $\overline\cP$ of a hereditary word property $\cP$  Lemma~\ref{lem:threeThirty} will allow us to ``switch'' from convergence in $d_{\Box}$ to  convergence in $d_1$, while
Lemma~\ref{prop:threeThirteen} shows  $d_1(\bw,\cP)\leq d_1(f_{\bw},\overline{\cP})$, thus allowing to pass from~$\overline{\cP}$ back to $\cP$.
With these results at hand we then give the proof of Theorem~\ref{theo:wordTesting} at the end of this section.

We start  with the following.

\begin{lemma}\label{lem:inclosure}
For all $\delta>0$ there is an $n_0$ such that  any $\bw\in\cP$ of length $n>n_0$ satisfies
\[d_{\Box}(f_{\bw},\overline\cP)=\inf_{g\in\overline\cP}d_{\Box}(f_{\bw},g)<\delta.\]
\end{lemma} 
\begin{proof}
Let $\delta>0$ be given and for a contradiction
 suppose there is a sequence
    $(\bw_{n})_{n\to\infty}$ in $\cP$ such that every $\bw_n$ satisfies
    $d_{\Box}(f_{\bw_{n}},\overline{\cP})\geq\delta$.
    By compactness of $(\cW,d_\Box)$ 
    there is an $f\in\cW$ such that (by passing to a subsequence\footnote{The term ``passing to a subsequence'' means considering a subsequence instead of the original sequence. However, to avoid making the notation more cumbersome, the subsequence keeps the same name as the original sequence.})
    $f_{\bw_{n}}\overset{\Box}{\to} f$. 
    In particular, 
     $(\bw_{n})_{n\to\infty}$ is a sequence in $\cP$ which converges to $f$ and thus $f\in\overline{\cP}$.
     However, this implies that
     $d_{\Box}(f_{\bw_{n}},\overline{\cP})\leq d_{\Box}(f_{\bw_{n}},f)<\delta$
    for large enough $n$, a contradiction.
\end{proof}

Recall that a property $\cP$ is hereditary if
$\sub(I,\bw)\in\cP$ for every $\bw\in\cP$ of length $n$ and every
$I\subseteq [n]$.
\begin{proposition}\label{prop:threeNine}
  If $\cP$ is a hereditary word property, then 
  \begin{align*}
  \overline{\cP}
 = \{ f\in\cW \colon\, \PP(\sub(\ell,f)\in\cP)=1\text{ for all }\ell\geq 1\}
  = \{ f\in\cW\colon\, t(\bu,f)=0 \text{ for all } \bu\not\in\cP\}.
  \end{align*}
  Moreover, if  $\cP$ does not contain all words,
  then every $f\in\overline{\cP}$ is
  $0$-$1$ valued except maybe on a set of null measure.
\end{proposition}
\begin{proof}
The second  equality holds since
 for each integer $\ell\geq 1$ we have
\begin{align}\label{eq:subvst}
 0= \PP(\sub(\ell,f)\not\in\cP) = \sum_{\bu\in\{0,1\}^{\ell}\setminus\cP}\PP(\sub(\ell,f)=\bu)
  =\sum_{\bu\in\{0,1\}^{\ell}\setminus \cP} t(\bu,f).
  \end{align}

To show the first equality recall from  Corollary~\ref{cor:randomwords} 
that $\big(\sub(\ell,f)\big)_{\ell\to\infty}$ converges to $f$ a.s.
Hence, if moreover $\PP(\sub(\ell,f)\in\cP)=1$ holds for every $\ell$,
then there is a sequence of words from $\cP$ which converges to $f$,
showing that $f\in\overline{\cP}$.
  For the converse, let
  $(\bw_n)_{n\to\infty}$ be a sequence of words in $\cP$ that converges to
  $f\in\overline{\cP}$, i.e., 
  $\lim_{n\to\infty}t(\bu,\bw_n)=t(\bu,f)$ for every word $\bu$.
  In particular, if  $\bu\not\in\cP$ then $t(\bu,\bw_n)=0$
  by heredity of $\cP$ and thus $t(\bu,f)=0$. By~\eqref{eq:subvst} we then obtain $\PP(\sub(\ell,f)\not\in\cP)=0$.

  Finally, suppose that $f\in\overline{\cP}$ and that there is a
  $\bu\in\{0,1\}^\ell\setminus\cP$ for some $\ell$.
  Let $\bX=(X_1,...,X_{\ell})$ be uniformly chosen in
  $[0,1]^\ell$, then the characterization of $\overline{\cP}$ and~\eqref{eq:subseqdensity} yields
  \begin{align*}
    0 =\PP(\sub(\ell,f) \not\in\cP)
    & \geq {t(\bu,f)= \ell!} \int_{\substack{x_1,...,x_\ell\in f^{-1}(]0,1[)\\ 0\leq x_1<...<x_\ell\leq 1}} \prod_{i\in [\ell]}f^{u_i}(x_i)\diff x_1...\diff x_{\ell}{.}
  \end{align*}
  Thus, {$f^{-1}(]0,1[)$ has} null Lebesgue measure.
\end{proof}


\begin{lemma}\label{prop:threeThirteen}
  If $\cP$
  is {a} hereditary word property and $\bw$ is a word, then
  $d_1(\bw,\cP)\leq d_1(f_{\bw},\overline{\cP})$.
\end{lemma}
\begin{proof}
  The claim is obvious if {$\overline{\cP}$ is empty or $\bw\in\cP$}.
  Let $\delta>0$ and let $\bw\not\in\cP$ be a word of length~$n$. By definition  there is a $g\in\overline{\cP}$ such that $d_1(f_{\bw},g)\leq d_1(f_{\bw},\overline{\cP})+\delta$ and
  by Proposition~\ref{prop:threeNine} $g$ is  $0$-$1$ valued
  and $\PP(\sub(n,g)\in\cP)=1$ for all $n\geq 1$. In particular, $\PP(\bw'\in \cP)=1$ when $\bw'=\sub(\bX,g)$, with $\bX=(X_1,...,X_n)$ and  $X_i$ is uniformly chosen in the interval $[\frac{i-1}n,\frac in]$.
  Since the probability {(conditioned on $X_i$)}
  that index $i$ contributes to
  $d_1(\bw,\bw')$ is $g(X_i)$ if $w_i=0$ and $1-g(X_i)$ if $w_i=1$ we have
  \[
  \EE(d_1(\bw,\bw'))=\|f_{\bw}-g\|_1=d_1(f_{\bw},g)\leq d_1(f_{\bw},\overline{\cP})+\delta.
  \]
  In particular, there exists  $\widetilde{\bw}\in\cP$ for which $d_1(f_{\bw},\overline\cP)+\delta\geq d_1(\bw,\widetilde{\bw})\geq d_1(\bw,\cP)$ holds.
  Since $\delta$ is arbitrary, the desired conclusion follows.
\end{proof}

\begin{lemma}\label{lem:threeThirty}
  If $\cP$
  is {a} hereditary word property and $(f_n)_{n\to\infty}$
  is a sequence of functions in~$\cW$ such that
  $d_{\Box}(f_n,\overline{\cP})\to 0$,  then  $d_1(f_n,\overline{\cP})\to 0$.
\end{lemma}
\begin{proof}
We may assume that $\cP$ is not the set of all words, otherwise $\overline{\cP}=\cW$ and the lemma follows.
Let $(f_n)_{n\to\infty}$ with $d_{\Box}(f_n,\overline{\cP})\to 0$ be given. By definition
 there is a sequence $(\eps_n)_{n\to\infty}$ that converges to $0$, and a sequence $(g_n)_{n\to\infty}$ in $\overline{\cP}$ such that 
  \[    d_\Box(f_n,g_n)\le d_\Box(f_n,\overline{\cP})+\eps_n\qquad\text{for all } n.
    \]
Since $(\cW,d_{\Box})$ is compact   we may assume (by passing to a subsequence)
 that $g_n\overset{\Box}{\to}f$ for some
{$f\in\cW$}. 
By definition of closure there is an increasing function $m=m(n)$ such that for every $g_n\in\overline{\cP}$ there is a word $\bw_m\in\cP$ with $d_{\Box}(f_{\bw_m},g_n)\leq\eps_n$.   
Since $d_{\Box}(f_{\bw_m},f)\leq d_{\Box}(f_{\bw_m},g_n)+d_{\Box}(g_n,f)\leq 2\eps_n\to 0$ when $n\to\infty$, it follows that $f_{\bw_m}\overset{\Box}{\to}f$ or 
in other words, that $(\bw_m)_{m\to\infty}$ converges to $f$ and thus $f\in\overline{\cP}$.
Moreover, by Proposition~\ref{prop:threeNine},
we get that $f$ is $0$--$1$ valued.
  Consider the Lebesgue measurable sets $\Omega_b=f^{-1}(b)$ for $b\in\{0,1\}$.
  Then
  \[
  d_1(f_n,f)=
  \|f_n-f\|_1 = \int_{\Omega_0}f_n+\int_{\Omega_1}(1-f_n).
  \]
  In case $\Omega_0, \Omega_1$  are intervals we conclude from
  $\lim_{n\to\infty}d_\Box(f_n,f)=0$ that
  \[
  \lim_{n\to\infty}\int_{\Omega_0}f_n=\int_{\Omega_0}f=0\qquad\text{and}\qquad \lim_{n\to\infty}\int_{\Omega_1}(1-f_n)=\int_{\Omega_1}(1-f)=0.
  \] 
  By standard limiting arguments this extends to finite unions of intervals
  and finally to all Lebesgue measurable sets, so we conclude
  that $d_1(f_n,f)\to 0$ when $n\to\infty$ and the lemma follows  since $f\in\overline{\cP}$.
\end{proof}

Finally, we are ready to derive the main result of this section, that any hereditary word property is testable.
\begin{proof}[Proof of Theorem~\ref{theo:wordTesting}]
  Let $\cP$ be a hereditary word property and $\overline\cP$ its closure. 
  Let $\cP'(\frac1i)$ be the collection
  of words $\bv$ that either belong to $\cP$ or  that satisfy
  $d_{\Box}(f_{\bv},\overline{\cP})\leq\frac1i$. Let $\cP'=\bigcap_{i\in\NN} \cP'(\frac1i)$ which  
  we claim to be a test property for $\cP$. 
  Perfect completeness is clearly satisfied since $\cP$ is hereditary, so
    when $\bw\in\cP$ then $\sub(\ell,\bw)\in \cP\subseteq \cP'$ with probability $1$.
    
    To prove soundness let $\eps>0$ be given. We need to show the existence of
     an $\ell(\eps)$ 
  such that  any $\bw\in\{0,1\}^n$ with $d_1(\bw,\cP)\geq \eps$ satisfies
   $\PP(\sub(\ell,\bw)\in\cP')\leq \tfrac{1}{3}$ for all
  $\ell(\eps)\leq\ell\leq n$.
    We apply  Lemma~\ref{lem:threeThirty} and conclude that there is a $\delta=\delta(\eps)>0$
  such that $d_1(f,\overline{\cP})<\eps$ whenever $d_{\Box}(f,\overline{\cP})<\delta$. Further, 
  we apply Lemma~\ref{lem:combRandomWords}  and
  Lemma~\ref{lem:inclosure} with~$\delta/4$ 
  to obtain $n_0$ and $\ell_0$.  Finally, choose $\ell(\eps)\geq\max\{n_0,\ell_0,2^{15}\eps^{-3}\}$.
  
  With this choice of constants note that a word $\bv\in\cP'$ of length $\ell\geq \ell(\eps)$ satisfies  $d_{\Box}(f_{\bv},\overline{\cP}\big)\leq\delta/2$. Indeed, by definition of $\cP'$
  this is clear if  $\bv\in \cP'\setminus \cP$ and for  $\bv\in \cP$
   this follows from Lemma~\ref{lem:inclosure} and  $\ell\geq \ell(\eps)$.
Now  let $\bw\in\{0,1\}^n$ with $d_1(\bw,\cP)\geq \eps$  be given, let $\bu=\sub(\ell,\bw)$  for some $\ell(\eps)\leq \ell\leq n$ and
  for a contradiction assume that  soundness does not hold. Then we conclude from the above that
    \[\PP\big(d_{\Box}(f_{\bu},\overline{\cP})\leq\delta/2\big)\geq \PP\big(\bu\in\cP'\big)>1/3.\]
 Further, by Lemma~\ref{lem:combRandomWords} and the choice of $\ell(\eps)$ we have 
 $\PP\big(d_{\Box}(f_{\bu},f_{\bw})<\delta/4\big)> 2/3$
 and thus, there is a word ${\bv}$  such that
  $d_{\Box}(f_{{\bv}},f_{\bw})<\delta/4$ and
  $d_{\Box}(f_{{\bv}},\overline{\cP})\leq\delta/2$ hold simultanously.  
  Triangle inequality then gives $d_{\Box}(f_{\bw},\overline{\cP})<\delta$, which
  by Lemma~\ref{lem:threeThirty} and the choice of $\delta$, implies
  $d_1(f_{\bw},\overline{\cP})<\eps$.
  Finally, Proposition~\ref{prop:threeThirteen}  yields 
  $d_{1}(\bw,\cP)\leq d_{1}(f_{\bw},\overline{\cP})<\eps$
  which is the desired contradiction.
\end{proof}

\section{Finite forcibility}\label{sec:forcible}
{In this section we investigate word limits that are prescribed by a finite number of subsequence densities.}
 {In particular, we prove Theorem~\ref{thm:forcible} showing that piecewise polynomial functions are  forcible. 
 The proof relies on the following lemma which shows, among other,
that  moments of cumulative distributions
  can be characterized by a finite number of subsequence densities of
  the distribution's mass density function.
\begin{lemma}\label{lem:xFx}
  If $f:[0,1]\to [0,1]$ is a Lebesgue measurable function
    and $F(x)=\int_0^x f(t)\diff t$, 
    then for each $i,j\in\NN$ we have
\[
\int x^iF(x)^jdx =
  \frac{i!j!}{(i+j+1)!}\sum_{\substack{\bu\in\{0,1\}^{i+j+1} \\ u_1+...+u_{i+j}\geq j}}{\tbinom{u_1+...+u_{i+j}}{j}}t(\bu,f).
\]
\end{lemma}
\begin{proof}
Observe that
\begin{align*}
  \int x^iF(x)^j\diff x 
  & =
  \int\Big(\int_0^x dy\Big)^i\Big(\int_0^x f(z)dz\Big)^j\diff x
  \\ & =
  \int\Big(\int_{0\leq y_1,...,y_i\leq x}\diff y_1...\diff y_i\Big)
      \Big(\int_{0\leq z_1,...,z_j\leq x}\prod_{k=1}^{j}f(z_k)\diff z_1...\diff z_j\Big)\diff x
  \\ & =
  i!j!\int\Big(\int_{0\leq y_1<...<y_i\leq x}\diff y_1...\diff y_i\Big)
      \Big(\int_{0\leq z_1<...<z_j\leq x}\prod_{k=1}^{j}f(z_k)\diff z_1...\diff z_j\Big)\diff x
  \\ & =
  i!j!\sum_{S\subseteq [i+j] : |S|= j}
  \int_{0\leq x_1<...<x_{i+j}\leq x}\prod_{s\in S}f(x_s)\diff x_1...\diff x_{i+j}\diff x.
\end{align*}
Since
\[
1=
\prod_{s\in [i+j]\setminus S} \big(f(x_s)+(1-f(x_s))\big)
=\sum_{U\subseteq [i+j]: S\subseteq U}\big(\prod_{s\in U\setminus S}f(x_s)\big)\big(\prod_{s\not\in U}(1-f(x_s))\big),
\]
we get
\begin{align*}  
  \int x^iF(x)^j\diff x 
  & = 
  i!j!\sum_{U\subseteq [i+j] : |U|\geq j}\tbinom{|U|}{j}
  \int_{0\leq x_1<...<x_{i+j}\leq x}\prod_{s\in U}f(x_s)\prod_{s\not\in U}(1-f(x_s))\diff x_1...\diff x_{i+j}\diff x  
  \\ & =
    \frac{i!j!}{(i+j+1)!}\sum_{\substack{\bu\in\{0,1\}^{i+j+1} \\ u_1+...+u_{i+j}\geq j}}\tbinom{u_1+...+u_{i+j}}{j}t(\bu,f).
\end{align*}
\end{proof}

We next prove this section's main result concerning the finite forcibility of piecewise polynomial functions.
\begin{proof}[Proof of Theorem~\ref{thm:forcible}]
 Let $f$ be a piecewise polynomial function given with the corresponding polynomials
  $P_1(x),\dots, P_k(x)$ and the corresponding intervals $(I_1,...,I_k)$ ordered by their natural appearance in $[0,1]$.
  Thus, 
for any $i\in[k]$ and $x\in I_i$ we have $f(x)=P_i(x)$.

 
For each $i\in [k]$ and $x\in I_i$ we define
  \[
  Q_i(x)=\int_{I_i\cap [0,x]} P_i(t)\diff t+\sum_{j=1}^{i-1} \int_{I_j} P_j(t)\diff t.
  \]
  Then $Q_i$ is a polynomial on the interval $I_i$ and we extend it to the whole interval $[0,1]$.
 Further, $F(x)=\int_0^xf(t)\diff t$ satisfies
  $F(x)=Q_i(x)$ for all $i\in[k]$ and $x\in I_i$.
 {Next, let} $d=\sum_{i\in[k]}\deg(Q_i)$ 
 and define the polynomial 
\[
P(x,y)=\big(y-Q_1(x)\big)^2\big(y-Q_2(x)\big)^2\dots \big(y-Q_k(x)\big)^2=\sum_{1\leq i+j\leq 2d}c_{ij} x^jy^{i}
\]
for some coefficients $c_{ij}$. 
Since $P(x,F(x))=0$ for all $x\in [0,1]$ we have 
 \begin{align}\label{eq:yF}\int_0^1P\big(x,F(x)\big)\diff x=0.\end{align} This remains true when we remove duplicated $Q_i$'s in the definition of $P(x,y)$, hence
we may  assume that the $Q_i$'s are pairwise distinct. 
  
By
Lemma~\ref{lem:xFx} we conclude that there is a list of words of length at most $2d+1$, say, $\bu_1,\dots, \bu_s$ with $s\leq 2^{2d+1}$, 
such that~\eqref{eq:yF} already follows from the prescription of the values $t(\bu_i,f)$, $i\in [s]$.
In particular, if   $h\in\cW$ satisfies $t(\bu_i,h)=t(\bu_i,f)$ for all
$i\in [s]$, then $H(x)=\int_0^x h(t)\diff t$ is continuous and
satisfies  $0=\int_0^1P\big(x,H(x)\big)\diff x$. Since $P\geq 0$ this implies that $P\big(x,H(x)\big)=0$  everywhere,
and by the definition of $P(x,y)$ we  conclude that for each $x\in[0,1]$ there is an $\ell=\ell_H(x)\in[k]$ such that $H(x)=Q_{\ell}(x)$.

Let $(a_1,b_1),(a_2,b_2),\dots, (a_t,b_t)$ be the intersection points of $Q_1,\dots, Q_k$ ordered by their first coordinate (with ties broken arbitrarily) and
let $a_0=0$ and $a_{t+1}=1$. 
Note that $t\leq \tbinom k2\max_{i\in[k]}\deg(Q_i)$ as two distinct polynomials $Q_i$ and~$Q_j$ have at most 
$\max \{\deg(Q_i),\deg (Q_j)\}$ intersection points.
Further, for an interval $(a_{i-1},a_{i})$,  $i\in[t+1]$,  the function $\ell_H(x)$ must be constant on this interval.
This is because $H$ is continuous and therefore if $x'>x$ and $\ell_H(x)\neq \ell_H(x')$, then there must exist an intersection point in the interval
$[x,x']$.
 We infer that $H$ is uniquely determined by the $(t+1)$ values $\ell_H(\cdot)\in[k]$ on the intervals $(a_{i-1},a_{i})$, $i\in[t+1]$.
Hence, there are at most~$k^{t+1}$ such functions $H$, implying at most that many functions
$h:[0,1]\to[0,1]$ such that $t(\bu_i,h)=t(\bu_i,f)$ for all
$i\in [s]$.


To finish the proof note that by uniqueness of word limits, see Theorem~\ref{thm:limits},  we can find for  each~$h$, which differs from $f$ by a non-zero measure set, a word 
$\bu_h$ such that $t(\bu_h,f)\neq t(\bu_h,h)$.
Thus,~$f$ is uniquely determined by the densities of at most $s+k^{t+1}\leq (k+1)^{2k^2(1+\max_i\deg(P_i))}$ words.
\end{proof}\begin{remark}\label{rem:forcibleqr}
The same proof for $k=1$ and $P_1(x)=a$ being constant yields an alternative  proof of the second part of Theorem~\ref{thm:quasirandom}.
In this case \[P\big(x,F(x)\big)=\big(F(x)-a x\big)^2=F(x)^2-2axF(x)+a^2x^2\] 
and by Lemma~\ref{lem:xFx}, the fact $\int_0^1 P\big(x,F(x)\big)\diff x=0$ is determined by densities of words of length three.
\end{remark}

\section{Permutons from words limits}\label{sec:permFromWords}
In this section we use our results concerning word limits to give an alternative proof of two key results  by 
Hoppen et al~\cite{HOPPEN} concerning permutons, limits of permutation sequences, see Proposition~\ref{perm:cauchy} and Theorem~\ref{thm:hoppen}. {Overall, our approach gives a simpler proof for the existence of permutons, Theorem~\ref{thm:hoppen}, due to the simpler objects (words and measurable transformations of the unit interval) on which our analysis is carried out.
Moreover, in Proposition~\ref{perm:cauchy} we give} a direct proof (avoiding compactness arguments) of the equivalence between  $t$-convergence and  convergence in the respective cut-distance{, which we believe is both technically original and of independent interest.}

For $n\in\mathbb N$ we write $\mathfrak S_n$ for the set of permutations of order $n$ and $\mathfrak S$ for the set of all finite permutations.
Also, for $\sigma\in\mathfrak S_n$ and $\tau\in\mathfrak S_k$ we write $\Lambda(\tau,\sigma)$ for the number of copies of $\tau$ in $\sigma$, that is, the number of $k$-tuples $1\le x_1<\dots<x_k\le n$ such that for every $i,j\in[k]$
\begin{equation*}\sigma(x_i)\le \sigma(x_j)\hspace{.5cm}\text{ iff }\hspace{.5cm}\tau(i)\le \tau(j).\end{equation*}
The density of copies of $\tau$ in $\sigma$, denoted by $t(\tau,\sigma)$, was defined as the probability that $\sigma$ restricted to a randomly chosen $k$-tuple of $[n]$ yields a copy of $\tau$, that is
\[t(\tau,\sigma)=\begin{cases} \binom{n}{k}^{-1}\Lambda(\tau,\sigma)& \text{ if } n\ge k,\\
0 & \text{ otherwise. }\end{cases}\]
Following~\cite[Definition~1.2]{HOPPEN}, a sequence $(\sigma_n)_{n\to\infty}$ of permutations, with $\sigma_n\in\mathfrak S_n$ for each $n\in\mathbb N$, is said to be convergent if 
$\lim_{n\to\infty}t(\tau,\sigma_n)$ exists for every permutation $\tau\in\mathfrak S$.
A permuton is a probability measure $\mu$ on the Borel $\sigma$-algebra on $[0,1]\times[0,1]$ that has uniform marginals, that is, for every measurable set $A\subseteq [0,1]$ one has
\begin{equation*}\mu(A\times [0,1])=\mu([0,1]\times A)=\lambda(A).
\end{equation*}
The collection of permutons is denoted by  $\mathcal Z$. It turns out that every permutation may be identified with a permuton  which preserves the sub-permutation densities. Indeed, given a permutation $\sigma\in\mathfrak S_n$ we define the permuton $\mu_\sigma$ associated to $\sigma$ in the following way.
First, for $i,j\in [n]$ define
\[B_{i,j}=B_i\times B_j\qquad \text{where}\qquad B_i=\begin{cases}\big[\tfrac{i-1}n,\tfrac in\big)&     \text{if $i\neq n$,} \\
\big[\tfrac{n-1}n,1\big]&     \text{otherwise.} \end{cases}\]
and note that $B_{i,j}$ has Lebesgue measure
 $\lambda^{(2)}(B_{i,j})=1/n^2$ for every $i,j\in[n]$. For every measurable set $E\subseteq[0,1]^2$ we let  
\[\mu_\sigma(E)=\sum_{i=1}^nn\lambda^{(2)}(B_{i,\sigma(i)}\cap E)=\int_{E}n\mathbf{1}\{\sigma(\lceil nx\rceil)=\lceil ny\rceil\}\diff x\diff y.\]
It is easy to see that $\mu_\sigma\in \mathcal Z$.

We next argue that the densities of sub-permutations is preserved by
$\mu_\sigma$.
First, let us explain what we mean by sub-permutation densities for a permuton. Given $\mu\in\mathcal Z$ and $k\in \mathbb N$, we sample $k$ points $(X_1,Y_1),\dots,(X_k,Y_k)$, where each $(X_i,Y_i)$ is sampled independently accordingly to $\mu$. Then, if $\sigma,\pi\in\mathfrak S_k$ are two permutations such that 
\[X_{\pi(1)}\le \dots\le X_{\pi(k)}\hspace{.5cm}\text{and}\hspace{.5cm}Y_{\sigma(1)}\le \dots\le Y_{\sigma(k)},\]
we define the random sub-permutation $\sub(k,\mu)\in\mathfrak S_k$ by $\sub(k,\mu)=\sigma^{-1}\pi$.

{Henceforth, let $\mu^{(k)}=\mu\otimes\dots\otimes\mu$ be the $k$-fold product measure on $([0,1]\times[0,1])^k$.}
Given a permutation $\tau\in\mathfrak S_k$, the density of $\tau$ in $\mu$, denoted by $t(\tau,\mu)$, is defined as the probability that $\sub(k,\mu)$ equals {$\tau$}, that is
\[
t(\tau,\mu)=k!\int \mathbf{1}\{x_1<\dots<x_k, y_{\tau^{-1}(1)}<\dots<y_{\tau^{-1}(k)}\}\diff\mu^{(k)}{.}
\]
 It is easily shown (see~\cite[Lemma 3.5]{HOPPEN} for a proof) that given any permutations $\sigma\in\mathfrak S_n$ and $\tau\in \mathfrak S_k$ we have
\begin{equation}\label{perm:eq1}|t(\tau,\sigma)-t(\tau,\mu_\sigma)|\le\binom k2\frac1n. 
\end{equation} 
In particular, \eqref{perm:eq1} implies that a sequence of permutations $(\sigma_n)_{n\to\infty}$ converges if and only if $(t(\tau,\mu_{\sigma_n}))_{n\to\infty}$ is convergent for every permutation $\tau\in\mathfrak S$, and thus we may talk about permutations and permutons as the ``same" object. We say that a sequence of permutons $(\mu_n)_{n\to\infty}$ is $t$-convergent if $(t(\tau,\mu_n))_{n\to\infty)}$ converges for every $\tau\in\mathfrak S$. 

As in the case of words one can define a metric $d_\Box$ on  $\mathcal Z$ so that for all $\tau\in\mathfrak S$ the maps $t(\tau,\cdot)$ are Lipschitz continuous with respect to $d_\Box$. Indeed, given two permutons $\mu,\nu\in\mathcal Z$ define 
\[d_\Box(\mu,\nu)=\sup_{I,J\subseteq [0,1]}|\mu(I\times J)-\nu (I\times J)|,\]
where the supremum is taken over all intervals in $[0,1]$.
Next, we establish that $t(\tau,\cdot)$ is Lipschitz continuous with respect to $d_\Box$ via the following result, which is the permuton analogue of Lemma~\ref{lem:tcount}.
%
\begin{lemma}\label{perm:lipschitz}Given a permutation $\tau\in \mathfrak S_k$, for all permutons $\mu,\nu\in\mathcal Z$  we have
	\[|t(\tau,\mu)-t(\tau,\nu)|\le k^2d_\Box(\mu,\nu).\]
\end{lemma}
\begin{proof}Given $\vec x,\vec y\in [0,1]^k$, we denote by $(\vec x,\vec y)$ the vector of pairs $(x_j,y_j)$ for $j\in[k]$. Define
	\begin{equation}E^{\tau}=\{(\vec{x},\vec y)\in([0,1]\times [0,1])^k:x_1<\dots<x_k,  y_{\tau^{-1}(1)}<\dots<y_{\tau^{-1}(k)}\}.
\end{equation}
Then, we have $t(\tau,\mu)=k!\mu^{(k)}(E^\tau)$ and $t(\tau,\nu)=k!\nu^{(k)}(E^\tau)$. For $j\in [k]$, let \[Q_j=\mu^{(j)}\otimes\nu^{(k-j)}-\mu^{(j-1)}\otimes\nu^{(k-j+1)}\] and note that
\[\frac{1}{k!}|t(\tau,\mu)-t(\tau,\nu)|=|\mu^{(k)}(E^\tau)-\nu^{(k)}(E^\tau)|=\Big|\sum_{j=1}^{k}Q_j(E^\tau)\Big|\le \sum_{j=1}^{k}|Q_j(E^\tau)|.  \]
Let $(\vec x,\vec y)\in([0,1]\times [0,1])^k$. We define 
\[E_j^\tau(\vec x,\vec y)=\begin{cases}[0,x_2]\times [0,y_{\tau^{-1}(2)}]&\mbox{ for }j=1,\\
[x_{j-1},x_{j+1}]\times [y_{\tau^{-1}(j-1)},y_{\tau^{-1}(j+1)}]&\mbox{ for } 2\le j\le k-1,\\
[x_{k-1},1]\times [y_{\tau^{-1}(k-1)},1]&\mbox{ for }j=k
\end{cases}\]
if  $x_1<\dots<x_{j-1}<x_{j+1}<\dots<x_k$ and  $y_{\tau^{-1}(1)}<\dots <y_{\tau^{-1}(j-1)}<y_{\tau^{-1}(j+1)}<\dots<y_{\tau^{-1}(k)}$, and $E_j^\tau(\vec x,\vec y)=\emptyset$ otherwise.
Thus $\big|{\mu(E^{\tau}_j(\vec x,\vec y))-\nu(E^{\tau}_j(\vec x,\vec y))}\big|\le d_\Box(\mu,\nu)$ for all $(\vec x,\vec y)$. Letting $\vec x_{-j}\in [0,1]^{k-1}$ be the vector obtained by removing $x_j$ from $\vec x\in[0,1]^k$, for $2\le j\le k-1$
we have that 
\begin{align*}
  |Q_j(E^\tau)| 
  & = \Big|\int{\big(\mu(E^{\tau}_j(\vec x,\vec y))-\nu(E^{\tau}_j(\vec x,\vec y))\big)}\diff\mu^{(j-1)}\otimes \nu^{(k-j)}(\vec x_{-j},\vec y_{-j})\Big|\\
  & \le \int\Big|{\mu(E^{\tau}_j(\vec x,\vec y))-\nu(E^{\tau}_j(\vec x,\vec y))}\Big|\diff\mu^{(j-1)}\otimes \nu^{(k-j)}(\vec x_{-j},\vec y_{-j})\\
  & \le \int_{x_1<\dots<x_{j-1}<x_{j+1}<\dots<x_k}\Big|{\mu(E^{\tau}_j(\vec x,\vec y))-\nu(E^{\tau}_j(\vec x,\vec y))}\Big|\diff\mu^{(j-1)}\otimes \nu^{(k-j)}(\vec x_{-j},\vec y_{-j})\\
  & \le \frac{1}{(k-1)!} d_\Box(\mu,\nu),
\end{align*}
and for $j=1$ and $j=k$ the same bound holds. Finally, summing for each $j\in[k]$ we obtain the bound.
\end{proof}
In Hoppen et al.~\cite{HOPPEN}, the compactness of $(\mathcal
Z,d_\Box)$ is established and, as a consequence, also the equivalence
between $t$-convergence and convergence in $d_{\Box}$.  In particular,
they prove that for every convergent sequence of permutations
$(\sigma_n)_{n\to\infty}$ there is a permuton $\mu\in\mathcal Z$ such
that $t(\tau,\sigma_n)\to t(\tau,\mu)$ for all $\tau\in\mathfrak
S$. The goal of this section is to give a new proof of these two
results by using a more direct approach based on
Theorem~\ref{thm:limits} and the permuton analogue of Proposition~\ref{prop:tconvboxconv} based on Bernstein polynomials.

We start  with  a permuton analogue of Lemma~\ref{lem:f=g}. 
\begin{lemma}\label{perm:poly}Let $\mu\in\mathcal Z$ be a permuton and let $i,j\in\mathbb N$. There exist a set $S_{i,j}$ of permutations of order $i+j+1$ and positive numbers $(C^{i,j}_\tau)_{\tau\in S_{i,j}}$ such that
  \[
  \int_{[0,1]^2} x^iy^j\diff\mu(x,y)=\sum_{\tau\in S_{i,j}}C^{i,j}_\tau t(\tau,\mu).
  \]
\end{lemma}
We discovered that a similar result was proved by Glebov, Grzesik, Klimo\v{s}ov\'a and Kr\'al'~\cite[Theorem 3]{GGKK}. As the proofs are rather different 
we decide to include our proof here.
\begin{proof}
  We proceed as in the proof of Lemma~\ref{lem:f=g}.
  First, since $\mu$ has uniform marginals we have that
	\begin{align*}x^i &= \Big(\int_{[0,x]\times[0,1]}\diff \mu(x',y')\Big)^i=\int_{[0,1]^{2i}} \mathbf{1}\{x_1,\dots,x_i\le x\}\diff\mu(x_1,y_1)\dots\diff\mu(x_i,y_i)
        \end{align*}
	and similarly
	\[y^j
        =\int_{[0,1]^{2j}} \mathbf{1}\{y_{i+1},\dots,y_{i+j}\le y\}\diff\mu(x_{i+1},y_{i+1})\dots\diff\mu(x_{i+j},y_{i+j}). \]
	Whence, setting
\[G_{U}(\vec x,x)=\mathbf{1}\{x_1,\dots,x_i\le x\}\prod_{u\in U}\mathbf{1}\{x_{i+u}\le x\}\prod_{u\not\in U}\mathbf{1}\{x< x_{i+u}\}
  \]
and
\[
H_S(\vec y,y)=\mathbf{1}\{y_{i+1},\dots,y_{i+j}\le y\}\prod_{s\in S}\mathbf{1}\{y_{s}\le y\}\prod_{s\not\in S}\mathbf{1}\{y< y_{s}\}{,}
  \]
by the Fubini--Tonelli theorem, we have
	\begin{align*}
          x^iy^j &= \int_{[0,1]^{2(i+j)}}\mathbf 1\{x_1,\dots,x_i\le x\}\mathbf 1\{y_{i+1},\dots,y_{i+j}\le y\}\diff \mu^{(i+j)}(\vec x,\vec y)\\
	&=\sum_{U\subseteq [j]}\sum_{S\subseteq [i]}\int_{[0,1]^{2(i+j)}}G_{U}(\vec x,x)H_S(\vec y,y)\diff \mu^{(i+j)}(\vec x,\vec y).
\end{align*}
Finally, by reordering the position of the coordinates below and above $x$, respectively, we have
\begin{align*}
  \int_{[0,1]^2}x^iy^j\diff\mu(x,y) 
  & =\sum_{k\in[j]}\sum_{\ell\in[i]}\binom jk\binom i\ell\frac{(i+k)!(j-k)!}{(i+j+1)!}\sum_{\sigma\in\mathfrak S_{i+j+1}: \sigma(i+k+1)\geq j+1}t(\sigma,\mu),
\end{align*}
{where given distinct values $x_1,...,x_{i+j}\in [0,1]$
  and a $k$ element set $U\subseteq [j]$
  the factor $(i+k)!$ represents
  all the possible orderings
  of the $(i+k)$ element set
  $\{x_m : m\in [i] \vee m-i\in U\}$, the factor $(j-k)!$ represents
  all possible orderings
  of the $(j-k)$ element set $\{x_m : m\in [j]\setminus U\}$, and the
  $(i+j+1)!$ term in the denominator comes from the definition of
  $t(\sigma,\mu)$ for $\sigma\in\mathfrak S_{i+j+1}$.}
\end{proof}
As pointed out in~\cite{KP13}, {the previous result} can be used to prove the uniqueness of the limit of a sequence of permutations as we did for limits of words by using Lemma~\ref{lem:f=g}. Indeed, suppose that $\mu,\nu\in\mathcal Z$ are two permutons such that $t(\sigma,\mu)=t(\sigma,\nu)$ for every finite permutation $\sigma\in\mathfrak S$. By Lemma~\ref{perm:poly} and the Stone--Weierstrass theorem we deduce that for every continuous function $h:[0,1]^2\to\mathbb R$ we have
\[\int_{[0,1]^2}h(x,y)\diff\mu(x,y)=\int_{[0,1]^2}h(x,y)\diff\nu(x,y),\]
which implies that $\mu=\nu$. On the other hand, Lemma~\ref{perm:poly} can also be used to establish the permuton analogue of Proposition~\ref{prop:tconvboxconv}, that $t$-convergence implies the convergence with respect to $d_\Box$.
\begin{proposition}\label{perm:cauchy}If $(\mu_n)_{n\to\infty}$ is a sequence in $\mathcal Z$ which is $t$-convergent, then it is a Cauchy sequence with respect to $d_\Box$. Moreover, if  $\mu_n\stackrel{t}\to \mu$ for some $\mu\in\mathcal Z$, then $\mu_n\stackrel{\Box}\to \mu.$\end{proposition}
\begin{proof}Let $\eps>0$ be fixed and let $r=\lceil (80/\eps)^4\rceil$.
  Let $S_{i,j}\subseteq \mathfrak S_{i+j+1}$ and
{$C^{i,j}_{\tau}$} be as in the statement of
  Lemma~\ref{perm:poly}, define
  $C=\max \{C^{i,j}_{\tau}\colon\tau\in S_{i,j}, i,j\le r\}$,
  and let 
	\[\delta=\frac{\eps}{C(2r+1)!2^{4r+3}}.\]
	Let $n_0$ be sufficiently large so that for all $n,m\ge n_0$ we have 
	\begin{equation}\label{eq:perm:cauchy}|t(\tau,\mu_n)-t(\tau,\mu_m)|\le \delta \quad \text{ for all }\tau\in \bigcup_{i\in [r]} \mathfrak S_{i}.
	\end{equation}
        Hence, for $i,j\le r$ and $\nu=\mu_n-\mu_m$, by Lemma~\ref{perm:poly}
        {and since $|\mathfrak S_{i+j+1}|\leq (2r+1)!$}, we have 
\[\Big|\int_{[0,1]^2}x^iy^j\diff\nu(x,y)\Big|=\Big|\sum_{\tau\in S_{i,j}}{C^{i,j}_{\tau}} (t(\tau,\mu_n)-t(\tau,\mu_m))\Big|\le C(2r+1)!\delta. \]
For $a,b\in[0,1]$, let $J_{a,b}=\mathbf 1_{[0,a]\times[0,b]}$ and let $j_a,j_b$ be the largest {integers} such that $\frac{j_a}r\le a$ and $\frac{j_b}r\le b$.
Recall that the Bernstein polynomial of $J_{a,b}$ is denoted by $B_{r,J_{a,b}}$
  and observe that
\begin{align*}
  \Big|\int B_{r,J_{a,b}}(x,y)\diff\nu(x,y)\Big|
  & \le \sum_{i=0}^{j_a}\sum_{j=0}^{j_b}\tbinom{r}{i}\tbinom{r}{j}\Big|\int x^i(1-x)^{r-i}y^j(1-y)^{r-j}\diff \nu(x,y)\Big|\\
  & \le \sum_{0\le i,j\le r}\sum_{k=0}^{r-i}\sum_{\ell=0}^{r-j}\tbinom{r}{i}\tbinom{r}{j}\tbinom{r-i}{k}\tbinom{r-j}{\ell}\Big|\int x^{i+k}y^{j+\ell}\diff\nu(x,y)\Big|\\
  & \le C2^{4r}(2r+1)!\delta.
\end{align*}
Now, by Lemma~\ref{lem:Bta2} we have
\begin{align*}
  \left|\nu([0,a]\times[0,b])\right|&=\left|\int \mathbf 1_{[0,a]\times [0,b]}(x,y)\diff\nu(x,y)\right| \\
	& \le \left|\int B_{r,J_{a,b}}(x,y)\diff\nu(x,y)\right|+\left|\int (\mathbf 1_{[0,a]\times [0,b]}(x,y)-B_{r,J_{a,b}}(x,y))\diff\nu(x,y)\right|\\
	& \leq C{2^{4r}}(2r+1)!\delta + (8r^{-1/4}+2r^{-1/2}),
	\end{align*}
where the last inequality follows since $\mu_n$ and $\mu_m$ have
uniform marginals.
Putting everything together, by our choice of {$r$, $\delta$ and $\nu$}, we have 
\[
  d_\Box(\mu_n,\mu_m) \le 4\sup_{a,b\in[0,1]}|\nu([0,a]\times[0,b])| 
  \le C{2^{4r+2}}(2r+1)!\delta + 40r^{-1/4} \leq \eps.
\]
For the second part just replace $\mu_m$ by $\mu$ in~\eqref{eq:perm:cauchy} and choose $\nu=\mu_n-\mu$. Then, repeat the above argument.
\end{proof}

We can now give the alternative proof of the result of Hoppen et al~\cite{HOPPEN} concerning the existence of a limit (permuton)
for a convergent permutation sequence. Note that this limit is unique as discussed right after the proof of Lemma~\ref{perm:poly}.
  \begin{theorem}[Hoppen et al.~{\cite[Theorem~1.6]{HOPPEN}}]\label{thm:hoppen}
    For every convergent 
    sequence of permutations $(\sigma_n)_{n\to\infty}$ there exists a permuton
    $\mu\in\mathcal Z$ such that $\sigma_n\stackrel{t}{\to} \mu$. 
\end{theorem}
  \begin{proof}Let $(\sigma_n)_{n\to\infty}$ be given and let $(\mu_n)_{n\to\infty}$ be the sequence of corresponding permutons. Given $x\in[0,1]$ and $n\in\mathbb N$, we define
    \[
    f_{n,x}(y)=\int_{0}^xn\mathbf{1}\{\sigma_n(\lceil nt\rceil)=\lceil ny\rceil\}\diff t \hspace{.5cm}\text{for all $y\in [0,1]$.}
    \]
    It is easy to see that 
	\begin{enumerate}[(i)]
		\item\label{permuton:1} $f_{n,x}(\cdot)\le f_{n,x'}(\cdot)$ a.e. for all $x\le x'$,
		
		\item\label{permuton:2} $f_{n,0}(\cdot)=0$ a.e. for all $n\in\mathbb N$, and
		\item\label{permuton:3} $f_{n,1}(\cdot)=1$ a.e. for all $n\in\mathbb N$.
	\end{enumerate}
	We claim that $(f_{n,x})_{n\to\infty}$ converges {in $d_{\Box}$} for all $x\in[0,1]$. Indeed, by Proposition~\ref{perm:cauchy}, $(\mu_n)_{n\to\infty}$ is a Cauchy sequence with respect to $d_\Box$, and for every interval $I\subseteq [0,1]$
	\[\Big|\int_{I}(f_{n,x}-f_{m,x})(t)\diff t\Big|=\big|\mu_{n}([0,x]\times I)-\mu_{m}([0,x]\times I)\big|\le d_\Box(\mu_n,\mu_m).\]
Thus $(f_{n,x})_{n\to\infty}$ is a Cauchy sequence in $(\mathcal W,d_\Box)$ and therefore{, by Theorem~\ref{cor:compact},} it has a limit $f_x\in \mathcal W$. Furthermore, {by the dominated convergence theorem,} for all $x\in[0,1]$ we have  
	\begin{equation}\label{permuton:marginal}\int_0^1f_x(t)\diff t=\lim_{n\to\infty}\int_{0}^1f_{n,x}(t)\diff t=\lim_{n\to\infty}\frac{\lceil{nx}\rceil}{n}=x\end{equation}
and, because of~\eqref{permuton:1}, for all $a,x,x'\in [0,1]$,
	\begin{equation}\label{permuton:cont} \Big|\int_0^a f_{x}(t)\diff t-\int_0^a f_{x'}(t)\diff t\Big|\le \Big|\int_{0}^{1}(f_x-f_{x'})(t)\diff t\Big|= |x-x'|.\end{equation}
Given $0\le a<b\le 1$ and $0\le c<d\le 1$, we set 
	\begin{equation}\label{permuton:mu}\tilde\mu([a,b)\times[c,d))=\int_{c}^df_{[a,b)}(t)\diff t,\end{equation}
	where $f_{[a,b)}(t)=(f_b-f_a)(t)$ for all $t\in[0,1]$. We also set $\tilde{\mu}([a,b)\times [c,1])=\tilde{\mu}([a,b)\times [c,1))$ and $\tilde{\mu}([a,1]\times [c,d))=\tilde{\mu}([a,1)\times [c,d))$ for all $0\le a<b\le 1$ and $0\le c<d\le 1$.
	
We claim that $\tilde\mu$ extends to a unique measure $\mu\in\mathcal Z$.  Let $\mathcal F$ be the semiring on $[0,1]\times [0,1]$ consisting of boxes of the form $[a,b)\times [c,d)$, $[a,1]\times [c,d)$, $[a,b)\times [c,1]$ and $[a,1]\times [c,1]$, for $0\le a<b\le 1$ and $0\le c<d\le 1$.
Due to~\eqref{permuton:1} and~\eqref{permuton:2} 
we see that $\tilde\mu\ge 0$ and that $\tilde\mu$ is monotone on $\mathcal F$. Moreover, it is clear that $\tilde\mu$ is finitely additive since $f_x(t)$ is integrable for all $x\in [0,1]$ and in the following we show that $\tilde\mu$ is $\sigma$-additive on $\mathcal F$. Let $(I_n\times J_n)_{n\in\mathbb N}$ be a sequence of pairwise disjoint boxes such that $\bigcup_{n\in\mathbb N}I_n\times J_n=I\times J\in\mathcal F$. Without loss of generality, we assume that $I_n=[a_n,b_n)$ and $J_n=[c_n,d_n)$ for all $n\in\mathbb N$. Since $\tilde\mu$ is monotone we have  $\sum_{i=0}^n\tilde\mu(I_i\times J_i)\le \tilde\mu(I\times J)$ for all $n\in\mathbb N$
and thus $\sum_{n\in\mathbb N}\tilde\mu(I_n\times J_n)\le \tilde\mu(I\times J)$. In order to prove the upper bound, let $\delta>0$ be arbitrary and define $I'_n=[a_n-2^{-n}\delta,b_n+2^{-n}\delta)\cap[0,1]$ and $J'_n=[c_n-2^{-n}\delta,d_n+2^{-n}\delta)\cap [0,1]$ for each $n\in\mathbb N$. Note that the closure of ${I\times J}$ is contained in $\bigcup_{n\in\mathbb N}I'_n\times J_n'$ and thus, as $[0,1]\times [0,1]$ is compact, there exists a finite covering $I\times J\subseteq (I'_{n_1}\times J'_{n_1})\cup \dots\cup (I'_{n_\ell}\times J'_{n_\ell})$. Observe that 
\[\tilde\mu(I'_{n_i}\times J'_{n_i})-\tilde\mu(I_{n_i}\times J_{n_i})=\displaystyle\int_{I'_{n_i}\setminus I_{n_i}}f_{J'_{n_i}}(t)\diff t+\int_{ I_{n_i}}(f_{J'_{n_i}}-f_{J_{n_i}})(t)\diff t
	\le 2\cdot 2^{-n_i}\delta+2\cdot 2^{-n_i}\delta,\]
where the inequality is due to $\|f_{J'_{n_i}}\|_{\infty}\leq 1$ and~\eqref{permuton:cont}. Then we have
\[\tilde\mu(I\times J)\le \sum_{i=0}^\ell \tilde\mu(I'_{n_i}\times J'_{n_i})\le \sum_{n\in\mathbb N}(\tilde\mu(I_{n}\times J_{n})+4\cdot 2^{-n}\delta)\le\sum_{n\in\mathbb N}\tilde\mu(I_{n}\times J_{n}) +4\delta, \]
which implies $\tilde\mu(I\times J)=\sum_{n\in\mathbb N}\tilde\mu(I_{n}\times J_{n})$ as $\delta>0$ was arbitrary. Therefore $\tilde\mu$ is a pre-measure on $\mathcal F$ and thus there exists a measure $\mu$ on the Borel sets extending $\tilde\mu$ (see Theorem 11.3 from~\cite{Billingsley}). Moreover, since $\tilde\mu$ is finite, it follows that $\mu$ is unique (see Theorem 10.3 from~\cite{Billingsley}). Finally, due to \eqref{permuton:3} we have  $f_{1}(\cdot)=1$ a.e. which, together with~\eqref{permuton:marginal}, imply that $\mu$ has uniform marginals and therefore $\mu\in\mathcal Z$.

        To conclude that $\sigma_n\stackrel{t}{\to} \mu$, we note that by Lemma~\ref{perm:lipschitz} it is enough to show that $d_\Box(\sigma_n,\mu)\to 0$. If not, then
        there is an $\eps >0$ and sequences $(x_n)_{n\to\infty}$ and $(a_n)_{n\to\infty}$ such that, without loss of generality, for all $n$ sufficiently large  we have
\[ \int_{0}^{a_n}f_{n,x_n}(t)\diff t \ge \mu([0,x_n)\times [0,a_n]) +\eps
  =\int_0^{a_n}f_{x_n}(t)\diff t+\eps.\]
Moreover, by compactness of $[0,1]$ we can find $a\in [0,1]$ such that (passing to a subsequence) $(a_n)_{n\to\infty}$ converges to $a$ and for all $n$ sufficiently large we have
\[
{\Big|\int_{0}^{a}f_{n,x_n}(t)\diff t
- \int_{0}^{a_n}f_{n,x_n}(t)\diff t\Big|\,,\,
\Big|\int_{0}^{a_n}f_{x_n}(t)\diff t
- \int_{0}^{a}f_{x_n}(t)\diff t\Big|
\leq |a-a_n|\leq \frac{\eps}{8}.}
\]
{Thus,}
\[
{\int_{0}^{a}f_{n,x_n}(t)\diff t-\int_0^{a}f_{x_n}(t)\diff t\geq \frac{3\eps}{4}}
\]
{Again by compactness, there exists an $x\in [0,1]$ such that
  (passing to a subsequence) $(x_n)_{n\to\infty}$ converges to $x$ and,
  by~\eqref{permuton:cont} applied with $x'=x_{n}$,
  for all $n$ sufficiently large we can assume that $\int_{0}^{a}(f_x-f_{x_n})(t)\diff t\leq\frac{\eps}{8}$.
Finally, observing that $\int_0^a(f_{n,x}-f_{n,x_n})(t)\diff t\leq |x-x_n|$, taking $n$ sufficiently large so that $|x-x_n|\leq\frac{\eps}{8}$ we conclude that}
\[
  \int_{0}^{a}f_{n,x}(t)\diff t\ge \int_0^af_x(t)\diff t+\frac\eps2,
\]
contradicting the fact that $(f_{n,x})_{n\to\infty}$ converges to $f_x$.
\end{proof}

\section{Extensions}\label{sec:extensions}
In this section we consider two generalizations of our limit theory for binary words.
First, to non-binary words, and then to higher dimensional array structures.

\subsection{Non-binary words.}
Let $\Sigma$ be a finite alphabet.  For a word $\bw\in\Sigma^n$ and
an interval $I\subseteq[n]$ let $N_a(\bw, I)$ denote the number of
occurrences of $a\in\Sigma$ in $\sub(I,\bw)$ and let
$N_a(\bw)=N_a(\bw,[n])$.  Moreover, as for the binary alphabet case,
denote by $\binom\bw\bu$ the number of subsequences of $\bw$ which
coincide with $\bu$ {and, assuming the length of $\bu$ is $\ell$, let $t(\bu,\bw)$ be the probability that a randomly chosen $\ell$-subsequence of $\bw$ yields a copy of $\bu$.}
A sequence $(\bw_n)_{n\to\infty}$ of words $\bw_n\in\Sigma^n$ is
called $o(1)$-uniform if for each $a\in\Sigma$ there is a density $d_a$
such that $N_a(\bw_n, I)= d_a|I|+o(1)n$ holds for each interval
$I\subseteq [n]$.
We obtain the following analogue (generalization) of Theorem~\ref{thm:quasirandom} for finite size alphabets.
\begin{theorem}\label{thm:extensCount}
Given a sequence $(\bw_n)_{n\to\infty}$ of words $\bw_n\in\Sigma^n$
over the finite size alphabet~$\Sigma$. If $(\bw_n)_{n\to\infty}$ is
$o(1)$-uniform, then {for each $a\in\Sigma$ there is a density $d_a\in [0,1]$} such that for every
$\ell\in\NN$ and every word $\bu\in\Sigma^\ell$ we have
{$\tbinom{\bw_n}\bu=\prod_{a\in\Sigma}d_a^{N_a(\bu)}\binom
  n\ell+o(n^\ell)$}.
Conversely, if for some {collection
of densities $\{d_a\in [0,1] : a\in\Sigma\}$}
we have {$\tbinom{\bw_n}\bu=\prod_{a\in\Sigma}d_a^{N_a(\bu)}\binom n3+o(n^3)$} for all words $\bu\in\Sigma^3$, then $(\bw_n)_{n\to\infty}$
is $o(1)$-uniform.
\end{theorem}
\begin{proof}
  The first part of the theorem follows from Remark~\ref{remark:alphabetcounting} {(choosing, for each $a\in\Sigma$, the $g^{a}$'s therein as the constant function $g^{a}=d_{a}$)} by an argument similar to the one used in {the first part of the proof} of Lemma~\ref{lem:tcount}.  For the second part, let us consider a letter $a\in\Sigma$ and a word $\bw$ over $\Sigma$. We define the binary word $\bw^{a}$ as the word obtained by replacing each letter $a$ in $\bw$ by $1$ and the remaining letters by $0$. Moreover, for $\bu\in\{0,1\}^\ell$ we let $\Sigma_{a}(\bu)$ be the set of words $\bv\in\Sigma^\ell$ such that $\bv^{a}=\bu$.
    Then, it is easy to see that
\begin{equation}\label{largeralphabet:convergence}
    t(\bu,\bw^a)=\sum_{\bv\in\Sigma_{a}(\bu)}t(\bv,\bw).
\end{equation}		
For each $a\in\Sigma$ we can thus define the sequence $(\bw_n^a)_{n\to\infty}$ of words over the alphabet $\{0,1\}$ which, because of~\eqref{largeralphabet:convergence} {and since $\sum_{b\neq a}d_b=1-d_a$}, satisfies the counting property for subsequences of length 3. From Theorem~\ref{thm:quasirandom} {and our working hypothesis} we conclude that $(\bw_n^a)_{n\to\infty}$ is $o(1)$-uniform over the alphabet $\{0,1\}$ and thus we deduce that $N_a(\bw_n, I)=N_1(\bw_n^a,I)= d_a|I|+o(1)n$ for all intervals $I\subseteq [n]$. By repeating the above argument for each letter in $\Sigma$ we conclude that $(\bw_n)_{n\to\infty}$ is $o(1)$-uniform.
\end{proof}

Similarly, one can obtain an analog of Theorem~\ref{thm:limits}
concerning limits of convergent word sequences for larger alphabets.
A sequence $(\bw_n)_{n\to\infty}$ of words over the
alphabet~$\Sigma=\{a_1,\dots,a_k\}$ is convergent if for all
{$\ell\in\NN$}
and $\bu\in\Sigma^\ell$ the subsequence density
$\left(\tbinom{\bw_n}\bu/ \tbinom n\ell\right)_{n\to\infty}$ converges. {Moreover, given a $k$-tuple of functions $\vecf=(f^{a_1},\dots,f^{a_k})\in\mathcal W^k$ such that $f^{a_1}(x)+\dots+f^{a_k}(x)=1$ for almost all $x\in [0,1]$,} we say that $(\bw_n)_{n\to\infty}$ converges to
$\vecf=(f^{a_1},\dots ,f^{a_k})$ if for all {$\ell\in\mathbb N$} and $\bu\in\Sigma^\ell$ the subsequence density $\left(\tbinom{\bw_n}\bu/ \tbinom
n\ell\right)_{n\to\infty}$ converges to
\[
t(\bu,\vecf)=\ell!\int_{0\leq x_1<\dots<x_\ell\leq 1}\prod_{i\in[\ell]}f^{u_i}(x_i)\diff x_1\dots \diff x_\ell.
\]
For the case of non-binary alphabets, we obtain the following limit theorem.  
\begin{theorem}[Limits of convergent {$k$-letter} word sequences]
  {Let  $\Sigma=\{a_1,\dots,a_k\}$.}
\begin{itemize}
\item Each convergent sequence $(\bw_n)_{n\to\infty}$ of words, $\bw_n\in\Sigma^n$, converges to some vector
$\vecf=(f^{a_1},\dots ,f^{a_k})\in\cW^k$ and $f^{a_1}(x)+\dots+f^{a_k}(x)=1$ for almost all $x\in[0,1]$.
Moreover, if  $(\bw_n)_{n\to\infty}$ converges to $\vecg=(g^{a_1},\dots ,g^{a_k})$, then {$f^{a_i}= g^{a_i}$ almost everywhere, for all $i\in[k]$.}
\item Conversely, for every  vector $\vecf=(f^{a_1},\dots ,f^{a_k})\in\cW^k$ which satisfies $f^{a_1}(x)+\dots+f^{a_k}(x)=1$ for almost all $x\in[0,1]$
there is a  sequence $(\bw_n)_{n\to\infty}$ of words $\bw_n\in\Sigma^n$  which converges to $\vecf$.
\end{itemize}
\end{theorem}
\begin{proof}The first part follows by reducing to the size two alphabet case. Indeed,  fix $a_i\in\Sigma$. For each $n\in\mathbb N$ we define the word $\bw_n^{a_i}$ as in the proof of Theorem~\ref{thm:extensCount} and thus we  obtain a sequence  $(\bw_n^{a_i})_{n\to\infty}$ of words over the binary alphabet, which we claim is convergent. Indeed, since $(\bw_n)_{n\to\infty}$ is convergent then each term in the RHS in~\eqref{largeralphabet:convergence} is convergent and thus $(t(\bu,\bw_n^{a_i}))_{n\to\infty}$ is convergent.  Therefore,
	Theorem~\ref{thm:limits} implies that $(\bw_n^{a_i})_{n\to\infty}$
	converges to a (unique) $f^{a_i}\in\cW$. 
	In particular, {by~\eqref{eq:countcount},} the sequence $(f_n^{a_i})_{n\to\infty}$
	of functions associated to $(\bw_n^{a_i})_{n\to\infty}$ satisfies $f_n^{a_i}\overset{t}\to f^{a_i}$ and Proposition~\ref{prop:tconvboxconv} implies that $f_n^{a_i}\overset{\Box}\to f^{a_i}$ as well.
	The $k$-letters analog of Lemma~\ref{lem:tcount}, see Remark~\ref{remark:alphabetcounting}, and the analog of~\eqref{eq:countcount} for $k$-letters\footnote{Note that the proof of~\eqref{eq:countcount} given in the footnote~\ref{foot:countcount} extends without change to the $k$-letters case.} then yields that $(\bw_n)_{n\to\infty}$  converges to $\vecf=(f^{a_1},\dots ,f^{a_k})$ and it is 
	not hard to see that $f^{a_1}(x)+\dots+f^{a_k}(x)=1$ for almost all $x\in[0,1]$.

        To prove the second part, we exhibit a sequence of words which converges to a given  $\vecf=(f^{a_1},\dots ,f^{a_k})$. Consider the $\vecf$-random letter $(X,Y)\in[0,1]\times \Sigma$ obtained  by choosing $X$ uniformly in $[0,1]$ and, conditioned on $X=x$, choosing $Y$ to be $a_i\in\Sigma$ with probability $f^{a_i}(x)$. {Next, for each positive integer $n$} choose $\vecf$-random letters $(X_1,Y_1),\dots (X_n,Y_n)$ and a permutation $\sigma:[n]\to [n]$ such that $X_{\sigma(1)}\le\dots\leq X_{\sigma(n)}$. Then, define the $\vecf$-random word $\bw_n=Y_{\sigma(1)}\dots Y_{\sigma(n)}$.
	By fixing a letter $a_i\in\Sigma$ and replacing the $\bw_n$'s by $\bw_n^{a_i}$'s as above we obtain a sequence of $f^{a_i}$-random words over size two alphabets whose associated functions converge in the interval-norm to $f^{a_i}$ a.s.~due  to Corollary~\ref{cor:randomwords}. Then, 
	the argument shown in Lemma~\ref{lem:tcount}, see Remark~\ref{remark:alphabetcounting}, and the $k$-letters analog 
	of~\eqref{eq:countcount} imply that the $\vecf$-random word sequence converges to $\vecf$.
\end{proof}

\subsection{Multidimensional arrays.}
For $n,d\ge 1$, a $d$-dimensional $\{0,1\}$-array, $d$-array for short, of size $n$ is a function $A:[n]^d\to\{0,1\}$ which labels each element of $[n]^d$ with a $0$ or $1$. Note that for $d=1$ a $1$-array of size $n$ is just an $n$-letter word, and for $d=2$ a $2$-dimensional array is just a $n$-by-$n$ zero-one matrix. In general, given $d\ge 1$ and $\vec m=(m_1,\dots,m_d)\in\mathbb N^d$ a $d$-array of index $\vec m$  is a labeling $B:[m_1]\times\dotsb\times[m_d]\to\{0,1\}$. As in the other cases considered so far, we need to say what will be the notion of sub-array.
First, consider the $d=2$ case, that is, the case of matrices. We say that a matrix $A$ contains a copy of a matrix $B$ if by deleting rows and columns from $A$ one ends with the matrix $B$. In other words, we say that $B\in \{0,1\}^{k\times m}$ is a sub-array of $A\in\{0,1\}^{n\times n}$ if there are indices $1\le i_1<\dots<i_k\le n$ and $1\le j_1<\dots<j_m\le n$ such that $A_{i_r,j_s}=B_{r,s}$ for all $r\in[k]$ and $s\in[m]$. For higher dimensional arrays the idea is similar. We say that a $d$-array $A$ of size $n$ contains a copy of a $d$-array $B$ of index $\vec m\in[n]^d$ if there exists a set of indices 
\[L=\{(i^1_{j_{1}},\dots,i^d_{j_{d}})\in [n]^d: j_1\in[m_1],\dots,j_d\in [m_d]\},\]
with $i^k_{1}<\dots<i^k_{m_k}$ for each $k\in[d]$, such that $A|_L=B$. We denote by $\binom{A}{B}$ the number of copies of $B$ in $A$ and write
$t(B,A)$ for the density of $B$ in $A$, i.e.,
\[ t(B,A)=\dfrac{\binom{A}{B}}{\binom{n}{m_1}\dots \binom{n}{m_d}}.\] 

As we did for words, we can define a notion of convergence for $d$-arrays in terms of sub-array densities. We say that a sequence $(A_n)_{n\to\infty}$ of $d$-arrays, with $A_n\in\{0,1\}^{[n]^d}$ for each $n\in\mathbb N$, is $t$-convergent if for every $d$-array $B$ the sequence $(t(B,A_n)))_{n\to\infty}$ converges. Along the same lines of the proof of Theorem~\ref{thm:limits}, one can show that $t$-convergence is ``equivalent" to a higher order interval-distance and thus one can prove that every $t$-convergent sequence of $d$-arrays $(A_n)_{n\to\infty}$ converges to a Lebesgue measurable function $f:[0,1]^d\to[0,1]$. Moreover, for every Lebesgue measurable function $f:[0,1]^d\to[0,1]$ there exists a sequence of $d$-arrays, which arise from a random sampling from $f$, that converges to $f$ a.s.

\section{Concluding remarks}\label{sec:final}
We conclude with a discussion on some potential future research directions.
A variety of applications use data structures and algorithms on
strings/words. In many settings, it is reasonable to assume
that strings are generated by a random source of known characteristics.
Several basic (generic) probabilistic models have been proposed and are
often encountered in the analysis of problems on words, among others;
memoryless Markov, mixing and ergodic sources (for a detailed
discussion see~\cite{SzpankowskiBook}). Our investigations suggest that
a new probabilistic model for generating strings under which to analyze
the behavior of algorithms on words is the random words from limits
model of Section 4.4 (i.e., for $f\in\cW$, the sequence of distributions
on words $(\sub(n,f))_{n\in\NN}$).
For instance, one may consider variants of classical long-standing
open problems on words such as the Longest Common Subsequence (LCS) problem,
for which  it was shown~\cite{CS75} in the mid 70's that 
two random words uniformly chosen in $\{0,1\}^n$
have a LCS of size proportional to $n$ plus low order terms.
The exact value of the proportionality constant remains unknown,
although good upper and lower bounds have been established~\cite{Lueker09}.
Generalizing this model, one may consider two random strings $\sub(n,f_1)$ and
$\sub(n,f_2)$ and
  ask for conditions on $f_1,f_2\in\cW$ so that the expected length of the
longest common subsequence is of size
$o(n)$.

\medskip
\textbf{Acknowledgments:}
We would like to thank Svante Janson, Yoshiharu Kohayakawa and Jaime San Mart\'{\i}n
for valuable discussions and suggestions.
We also would like to thank an anonymous referee for detailed comments on a previous version of the paper.

A preliminary version of this
work will appear in Proceedings of the Latin American Symposium on
Theoretical Computer Science, LATIN 2020.

\bibliographystyle{acm}
\bibliography{biblio} 

\appendix
\section{}\label{sec:appendix}
In this section we give an alternative proof of Theorem~\ref{cor:compact} based on the regularity lemma for words which was introduced by Axenovich, Puzynina and Person in~\cite{APP13} to study decomposition of words into identical subsequences. For completeness, we give an (analytic) proof of the regularity lemma.

 A measurable partition $\cP$ of $[0,1]$ is a partition in which each atom is a measurable set of positive measure. Moreover, we say that $\cP$ is an interval partition if every atom in $\cP$ is a non-degenerate interval. In what follows, we will only consider measurable partitions $\cP$ with a finite number of atoms which we denote by $|\cP|$.
Given two partitions $\cP$ and $\cQ$ we say that $\cQ$ refines $\cP$, which we denote by $\cQ\preceq\cP$, if for every  $P\in \cP$ there are atoms $Q_1,\dots, Q_k\in\cQ$ such that $P=Q_1\cup \dots\cup Q_k$.  The common refinement of $\cP$ and $\cQ$ is the partition 
\[
\cP\wedge \cQ
  =\{A\cap B: A\in\cP, B\in\cQ\text{ such that }A\cap B\not=\emptyset \}.
\]
Moreover, given a measurable set $A$ we define the refinement of $\cP$ by $A$ as the common refinement of $\cP$ and the partition $\{A,A^c\}$. 

Let $f:[0,1]\to\mathbb{R}$ be an {integrable} function and let $\cP$ be a partition. As usual, let $\lambda$ denote the Lebesgue measure on $\mathbb R$. The conditional expectation of $f$ with respect to $\cP$ is the function $\EE(f|\cP)$ defined as
\[\EE(f|\cP)(x)=\sum_{P\in\cP}\frac{\vecone_{P(x)}}{\lambda(P)}\int_Pf(t)\diff t,
\]
for all $x\in[0,1]$. The energy of $\cP$ with respect to $f$ is defined by 
\[
\mathcal E_f(\cP)=\int_0^1\big(\EE(f|\cP)(x)\big)^2\diff x.
\]
Note that $\mathcal E_f(\cP)\leq \|f\|_\infty^2$.
The following is a well known (and easily derived)
result about conditional expectations.  
\begin{lemma}\label{lem:cond}
  Let $\cP$ and $\cQ$ be two partitions such that $\cQ\preceq\cP$.
  Given any {integrable}
  function $f:[0,1]\to\RR$, we have 
  \[\pushQED{\qed} 
  \int_0^1\EE (f|\cP)(t)\EE (f|\cQ)(t)\diff t=\int_0^1\big(\EE (f|\cP)(t)\big)^2\diff t.\qedhere
  \popQED
  \]
\end{lemma}

Our next result shows that every {$[0,1]$-valued} {integrable} function {over the interval $[0,1]$} can be approximated by a step function, which is supported on a partition of ``bounded complexity"{ (a somewhat related result by Feige et al., the so called Local Repetition Lemma, was obtained in~\cite[Lemma~2.4]{FKT17})}.

\begin{theorem}(Weak regularity lemma)\label{thm:reg} Let $\eps>0$ and let $\cP$ be an interval partition  of $[0,1]$. For every {integrable} function $f:[0,1]\to\mathbb [0,1]$ there exists an interval partition $\cP_\eps\preceq\cP$ such that  $\|f-\EE(f|\cP_\eps)\|_\Box\le \eps$  and $|\cP_\eps|\le |\cP|+2{\eps^{-2}}$.
\end{theorem}
\begin{proof}Set $\cP_1=\cP$ and suppose that $\|f-\EE(f|\cP_1)\|_\Box>\eps$, as otherwise the result is trivial. For $k\ge 1$, assume we have defined a sequence of interval partitions $\cP_k\preceq\dots\preceq\cP_1$ such that $\|f-\EE(f|\cP_k)\|_\Box>\eps$. This implies that there is an interval $I_{k+1}\not\in\cP_k$ such that
  \begin{equation}\label{eq:reg}
    \Big|\int_{I_{k+1}} (f-\EE(f|\cP_k))(t)\diff t\Big|>\eps.
  \end{equation}
  Define $\cP_{k+1}$ as the refinement of $\cP$ by $I_{k+1}$.
  Since either $I_{k+1}$ can split two distinct
    intervals of $\cP_k$ into two subintervals each, or split a single interval
    of $\cP_k$ into three subintervals, we have
   $|\cP_{k+1}|\le |\cP_{k}|+2$.
  From~\eqref{eq:reg} and by the Cauchy-Schwarz inequality, we deduce that 
\begin{align*}
  \eps^2 & <
  \left(\int_{I_{k+1}}\big(\EE(f|\cP_{k+1})(t)-\EE(f|\cP_k)(t)\big)\diff t\right)^2 \\
  & \le
  \int_0^1\Big(\EE(f|\cP_{k+1})(t)-\EE(f|\cP_k)(t)\Big)^2\diff t \\
  & = \int_{0}^1\Big(\EE(f|\cP_{k+1})(t)\Big)^2\diff t
  -\int_{0}^1\Big(\EE(f|\cP_{k})(t)\Big)^2\diff t,
\end{align*}
where the last equality follows from Lemma~\ref{lem:cond}. Thus we have
\[
1\ge\|f\|_{\infty}^2\geq \mathcal E_f(\cP_{k+1})\ge \mathcal E_f(\cP_{k})+\eps^2,
\]
and so, after at most $\eps^{-2}$ iterations, one finds some $\ell\le \eps^{-2}+1$ which satisfies $\|f-\EE(f|\cP_\ell)\|_\Box\le\eps$.
Since $|\cP_k|\le |\cP_{k+1}|+2$
for every $k\in[\ell]$, we get the claimed upper bound for $|\cP_\ell|$.
\end{proof}

\begin{lemma}[Theorem~35.5 from~\cite{Billingsley}]\label{lem:martingale}Let $f:[0,1]\to\mathbb R$ be an integrable function, and let $(\cP_i)_{i\in\NN}$ be a sequence of partitions such that $\cP_{i+1}\preceq\cP_i$ for all $i\in\NN$. Then the sequence $(\EE(f|\cP_i))_{i\in\NN}$ converges a.e.~to $\EE(f|\cP_\infty)$, where $\cP_\infty$ is the smallest $\sigma$-algebra containing each atom in $(\cP_i)_{i\in\NN}$.
\end{lemma}
{Before providing an alternative proof of Theorem~\ref{cor:compact}, we state some basic results from functional analysis. Given a normed vector space $(X,\|\cdot\|)$, the dual space $X^*$ of $X$ is the vector space of all linear and continuos functions from $X$ to $\mathbb R$. It turns out that $X^*$ is a normed vector space endowed with the operator norm $\|\varphi\|^*=\sup\{|\varphi(x)|\colon\|x\|=1\}$. On the other hand, the weak$^*$ topology on~$X^*$ is defined as the smallest topology that makes the functionals $\varphi\mapsto \varphi (x)$ continuous for all $x\in X$. In particular, a sequence $(\varphi_n)_{n\in\mathbb N}\subseteq X^*$ converges to $\varphi\in X^*$ in the weak$^*$ topology if and only if $\varphi_n(x)\to \varphi(x)$ for all $x\in X$. One of the main reasons to use the weak$^*$ topology instead of the operator norm topology is the Banach--Alaoglu theorem (see Theorem~5.18 from~\cite{Folland}), which states that the unit ball $B=\{\varphi\in X^*\colon\|\varphi\|^*\le 1\}$ is compact in the weak$^*$ topology.}

{A classical result in functional analysis states that $L^1([0,1])^*$ is isomorphic to $L^{\infty}([0,1])$ (see Theorem~ 6.15 from~\cite{Folland}). Thus, a sequence $(f_n)_{n\in\mathbb N}$ in $L^{\infty}([0,1])$ converges in the weak$^*$ topology if for every $g\in L^1([0,1])$ we have
\begin{equation}\label{weak:convergence}\lim_{n\to\infty}\int_0^1f_n(x)g(x)\diff x=\int_0^1f(x)g(x)\diff x.
\end{equation}  
It is easily shown that $\mathcal W$ is a weak$^*$ closed subset of the unit ball in $L^{\infty}([0,1])$, i.e., if a sequence $(f_n)_{n\in\mathbb N}$ in $\mathcal W$ converges to $f\in L^{\infty}([0,1])$ in the weak$^*$ topology then $f\in\mathcal W$. Indeed, letting $U_\eps=\{x\in[0,1]\colon f(x)\ge 1+\eps\}$ and $g=\vecone_{U_\eps}$ in~\eqref{weak:convergence} we see that $(1+\eps)\lambda(U_\eps)\le \lambda (U_\eps)$ which implies that $\lambda(U_\eps)=0$ for any $\eps >0$ and so $f(x)\le 1$ almost everywhere. On the other hand, letting $V_\eps=\{x\in[0,1]\colon f(x)\le -\eps\}$ and $g=\vecone_{V_\eps}$ we have $0\le -\eps\lambda(V_\eps)$ which implies $\lambda (V_\eps)=0$ for any given $\eps>0$, and thus $0\le f(x)\le 1$ almost everywhere.}

With these facts at hand we now give an alternative proof of the compactness of $(\cW,d_{\Box})$.
\begin{proof}[Proof of Theorem~\ref{cor:compact}]
  Let $(f_n)_{n\in\NN}$ be any sequence in $\mathcal W$. Since $\mathcal W$ is a weak$^*$ closed subset of the unit ball in $L^{\infty}([0,1])$, by the Banach--Alaoglu theorem we may assume that $(f_n)_{n\in\NN}$ converges in the weak$^*$ topology to some $f\in \mathcal W$. We claim that there are a collection of subsequences $(f_{n,k})_{n\in\NN}$, for $k\in \NN$, satisfying the following properties.

  \begin{enumerate}[(i)]
  \item {$f_{n,0}=f_n$ for all $n\in\NN$ and $\cP_0=\{[0,1]\}$.}
  \item\label{I:1} {For every $k\geq 1$, the sequence}
    $(f_{n,k})_{n\in\NN}$ is a subsequence of
    $(f_{n,k-1})_{n\in\NN}$.
		
\item\label{I:2} For $k\ge 1$, there is an interval partition
  $\cP_k\preceq \cP_{k-1}$ such that $|\cP_k|\le 3k^3$ and 
  $\|f_{n,k}-\EE(f_{n,k}|\cP_k)\|_\Box\le\tfrac{1}{k}$ for every $n\in\NN$. 
		
\item\label{I:3} For all $k\ge 1$, the sequence
  $(\EE(f_{n,k}|\cP_k))_{n\in\NN}$ converges a.e.~to $f^\ast_k=\EE(f|\cP_k)$.
  \end{enumerate}
 Clearly,
 properties \eqref{I:1}, \eqref{I:2}, and~\eqref{I:3} hold vacuously for $k=0$. Assume we have constructed the sequence up to step $k$. We apply
Theorem~\ref{thm:reg}, with $\eps=\tfrac{1}{k+1}$ and initial
partition $\cP_k$, to each function in the sequence $(f_{n,k})_{n\in\NN}$ so that for
every $n\in\NN$ we get an interval partition $\cP_{n,k}\preceq
\cP_{k}$, with $|\cP_{n,k}|\le |\cP_k|+2(k+1)^2\le 3(k+1)^3$ and such that
$\|f_{n,k}-\EE(f_{n,k}|\cP_{n,k})\|_\Box\le\tfrac{1}{k+1}$. For
$n\in\NN$, let $J_{n,k}=\{a_{n,1}=0<\dots<a_{n,\ell_n}=1\}$ be the set
of points that define the intervals of $\cP_{n,k}$. Note that
$\ell_n\le 1+3(k+1)^3$. By the pigeonhole principle there is an integer
$\ell\le 1+3(k+1)^3$ and a subsequence $(f_{n,k+1})_{n\in\NN}$ of $(f_{n,k})_{n\in\NN}$ such that
$\ell_{n}=\ell$ for all $n\in\NN$. Moreover, since $[0,1]$ is compact
we may even assume that $a_{n,i}\to a_i$ for each $i\in[\ell]$, where
$a_1=0\le\dots\le a_\ell=1$. Let $\cP_{k+1}\preceq \cP_k$ be the
partition defined by $J_k=\{a_1\le\dots\le a_\ell\}=\{a'_1<\dots <a'_{\ell'}\}$ for some $\ell'\le \ell$. Note that~\eqref{I:1}
and~\eqref{I:2} hold because of the definition of
$(f_{n,k+1})_{n\in\NN}$. Furthermore, because $\cP_{k+1}$ is finite
and since $(f_{n,k+1})_{n\in\NN}$ converges in the weak$^*$ topology to $f$ we conclude
that \eqref{I:3} also holds. 

Finally, by
Lemma~\ref{lem:martingale} we deduce that the sequence
$(f^\ast_k)_{k\in\NN}$ converges a.e.~to
$f_\infty=\EE(f|\cP_\infty)$. We claim that
$\lim_{k\to\infty}d_\Box(f_{k,k},f_\infty)= 0$. Indeed,
{given} $\eta>0$, using the dominated convergence theorem, \eqref{I:2} and~\eqref{I:3}, we have  for large $k\ge m\ge 3\eta^{-1}$ 
\[ d_\Box(f_\infty,f_{k,k})
\le d_\Box(f_\infty,f^\ast_m)
+ d_\Box(f^\ast_m,\EE(f_{k,k}|\cP_m))
+d_\Box(\EE(f_{k,k}|\cP_m),f_{k,k})
  \le \frac{\eta}{3}+\frac 1m+\frac{\eta}{3}
  \le\eta,
\]
{where the second inequality follow from the fact that $(f_{k,k})_{k\ge m}$ is a subsequence of $(f_{n,m})_{n\in\mathbb N}$.}
\end{proof}

\end{document}